\documentclass[letterpaper,11pt]{amsart}
\usepackage{amsmath,amssymb,amsthm,pinlabel,tikz,hyperref,mathrsfs,color}
\usepackage{verbatim}
\newcommand{\nc}{\newcommand}
\nc{\dmo}{\DeclareMathOperator}
\dmo{\ra}{\rightarrow}
\dmo{\Cl}{Cl}
\dmo{\Int}{Int}
\dmo{\Image}{Im}

\newcommand{\N}{\mathbb{N}}
\newcommand{\Z}{\mathbb{Z}}

\newcommand{\R}{\mathbb{R}}

\newcommand{\Homeo}{\mathrm{Homeo}}
\newcommand{\SHomeo}{\mathrm{SHomeo}}
\newcommand{\MCG}{\mathrm{MCG}}
\newcommand{\SW}{\mathrm{SW}}
\newcommand{\SB}{\mathrm{SB}}

\newcommand{\Fix}{\mathrm{Fix}}
\newcommand{\id}{\mathrm{id}}

\usetikzlibrary{decorations.markings}
\tikzset{->-/.style={decoration={
  markings,
  mark=at position #1 with {\arrow[scale=3]{>}}},postaction={decorate}}}
  
\nc{\nt}{\newtheorem}

\nt{theorem}{Theorem}

\newtheorem{thm}{{\bf Theorem}}[section]
\newtheorem{lem}[thm]{{\bf Lemma}}
\newtheorem{cor}[thm]{{\bf Corollary}}
\newtheorem{prop}[thm]{{\bf Proposition}}

\newtheorem{claim}[thm]{Claim} 
\newtheorem{remark}[thm]{Remark}

\newtheorem{ques}[thm]{Question}
\newtheorem{question}[thm]{Question}

\numberwithin{equation}{section}
\nt*{theorem_nonumber}{Theorem}

\theoremstyle{definition}
\newtheorem{ex}[thm]{Example}

\title[Goeritz groups of bridge decompositions]
{Goeritz groups of bridge decompositions}

\author[S. Hirose]{%
Susumu Hirose
}
\address{%
Department of Mathematics,  
Faculty of Science and Technology, 
Tokyo University of Science, 
Noda, Chiba, 278-8510, Japan}
\email{%
hirose\_susumu@ma.noda.tus.ac.jp
}

\author[D. Iguchi]{%
Daiki Iguchi
}
\address{%
Department of Mathematics, 
Hiroshima University, 1-3-1 Kagamiyama, Higashi-Hiroshima, 739-8526, Japan}
\email{%
d200643@hiroshima-u.ac.jp
}

\author[E. Kin]{%
    Eiko Kin
}
\address{%
       Department of Mathematics, Graduate School of Science, Osaka University Toyonaka, Osaka 560-0043, JAPAN
}
\email{%
        kin@math.sci.osaka-u.ac.jp
}

\author[Y. Koda]{%
Yuya Koda
}
\address{%
Department of Mathematics, 
Hiroshima University, 1-3-1 Kagamiyama, Higashi-Hiroshima, 739-8526, Japan}
\email{%
ykoda@hiroshima-u.ac.jp
} 

\makeatletter
\@namedef{subjclassname@2020}{%
  \textup{2020} Mathematics Subject Classification}
\makeatother

\subjclass[2020]{
Primary 57K20, Secondary 57K10, 57M12, 37B40, 37E30}

\keywords{%
mapping class group, bridge decomposition, curve complex, pseudo-Anosov, entropy.}

\thanks{%
S. H. is supported by JSPS KAKENHI Grant
 Numbers JP16K05156 and JP20K03618. 
E. K. is supported by JSPS KAKENHI Grant
 Number JP18K03299. 
Y. K. is supported by JSPS KAKENHI Grant
 Numbers JP17K05254, JP17H06463, JP20K03588 
 and JST CREST Grant Number JPMJCR17J4.}

\begin{document}

\maketitle

\begin{abstract}
For a bridge decomposition of a link in the $3$-sphere, 
we define the Goeritz group to be 
the group of isotopy classes of orientation-preserving homeomorphisms 
of the $3$-sphere that preserve each of the bridge sphere and link setwise. 
After describing basic properties of this group, we 
discuss the asymptotic behavior of the minimal pseudo-Anosov
entropies. 
This gives an application to the asymptotic behavior of 
the minimal entropies for the original Goeritz groups of Heegaard splittings 
of the $3$-sphere and the real projective space. 
\end{abstract}

%
%

\section*{Introduction}	
\label{section_introduction}

Every closed orientable 3-manifold $M$ can be decomposed  
into two handlebodies $V^+$ and $V^-$ by cutting it along a closed orientable surface $\Sigma$ of genus $g$ for 
some $g \geq 0$. 
Such a decomposition is called a {\it genus-$g$ Heegaard splitting} of $M$ and denoted by $(M; \Sigma)$. 
The {\it Goeritz group} $\mathcal{G} (M; \Sigma)$  
of the Heegaard splitting $(M; \Sigma) $
is then defined to be the group of isotopy classes of orientation-preserving self-homeomorphisms 
of $M$ that preserve each of the two handlebodies setwise. 
We note that this group can naturally be thought of as 
a subgroup of the mapping class group of the surface $\Sigma$. 
Indeed, restricting the maps in concern to $\Sigma$, we can describe the Goeritz group $\mathcal{G} (M; \Sigma)$ as 
$$ \mathcal{G} (M; \Sigma) = \MCG  (V^+) \cap  \MCG  (V^-) <  \MCG  (\Sigma), $$ 
where $ \MCG  ( ~ \cdot ~)$ is the mapping class group. 
The structure of this group is studied by many authors
(see e.g. recent papers 
\cite{JohnsonRubinstein13, JohnsonMcCullough13, FreedmanScharlemann18, ChoKoda19, IguchiKoda19, Zupan19} and references therein). 

In this paper, we define an analogous group, which we also call the 
{\it Goeritz group}, for a bridge decomposition of 
a link, and study some of its interesting properties. 
Recall that an {\it $n$-bridge decomposition} $(L; S)$ of a link $L$ in the 
3-sphere $S^3$ is a splitting of 
$(S^3, L)$ into two trivial $n$-tangles $(B^+, B^+ \cap L)$ and $(B^-, B^- \cap L)$ 
along a sphere $S \subset S^3$. 
We define the Goeritz group $\mathcal{G} (L; S)$ of the bridge decomposition $(L; S)$ to be the group of 
isotopy classes of orientation-preserving self-homeomorphisms of $S^3$ that 
preserve each of the trivial tangles setwise. 
(See Section \ref{section_Goeritz groups of bridge decompositions} for the 
definition.) 
Similarly to the case of Heegaard splittings, this group can be thought of as a 
subgroup of the mapping class group $ \MCG  (S, S \cap L)$, 
that is, the group of isotopy classes of orientation-preserving self-homeomorphisms of $S$ 
that preserve $S \cap L$ setwise. 
More precisely, restricting the maps in concern to $S$, 
the Goeritz group $\mathcal{G} (L; S)$ can be written as 
$$  \mathcal{G} (L; S) =  \MCG  (B^+, B^+ \cap L) \cap  \MCG  (B^-, B^- \cap L) <  \MCG  (S, S \cap L), $$
where $ \MCG  (B^\pm, B^\pm \cap L)$ is the group of isotopy classes of orientation-preserving 
self-homeomorphisms of the $3$-ball $B^\pm$ that preserve $B^\pm \cap L$ setwise.

The Goeritz groups of Heegaard splittings and those of bridge decompositions are 
related as follows. 
Given an $n$-bridge decomposition 
$$ (S^3, L) = (B^+, B^+ \cap L) \cup_S (B^-, B^- \cap L) $$
with $n \geq 2$, let  $q: M_L \to S^3$  be the $2$-fold covering of $S^3$ branched over $L$. 
Then that bridge decomposition induces a Heegaard splitting 
$M_L = V^+ \cup_{\Sigma} \cup V^- $, where $V^{\pm} := q^{-1} (B^{\pm})$ and $\Sigma := q^{-1} (S)$. 
Let $T$ be the non-trivial deck transformation of the covering. 
Note that $T|_\Sigma$ is a hyperelliptic involution of $\Sigma$. 
We define the {\it hyperelliptic Goeritz group} $\mathcal{HG}_T (M_L ; \Sigma)$ to be the group of orientation-preserving, fiber-preserving self-homeomorphisms 
of $M_L$ that preserve each of the two handlebodies setwise modulo isotopy through 
fiber-preserving homeomorphisms. 
Here a self-homeomorphism of $M_L$ is said to be {\it fiber-preserving} if 
it takes the fiber (with respect to the projection $q$) of each point in $S^3$ to the fiber of some point in $S^3$. 
Then we prove in Section \ref{subsec:Relation to hyperelliptic Goeritz groups of Heegaard splittings} 
the following theorem,  
which is the ``hyerelliptic Goeritz group version" of the original 
Birman-Hilden's theorem \cite{BirmanHilden71} for hyperelliptic mapping class groups. 
\begin{thm}
\label{introthm:HG(M;V)/i}
We have $ \mathcal{HG}_{T} (M_L ; \Sigma) / \langle [T] \rangle \cong \mathcal{G} (L; S) $. 
\end{thm}
It is then natural to think about a relationship between $\mathcal{HG}_{T} 
(M_L ;\Sigma)$ 
and the usual Goeritz group $\mathcal{G} (M_L ; \Sigma)$. 
In the same section, we give a rigorous proof of the following naturally expected result. 
\begin{thm}
\label{introthm:hyperelliptic Goeritz group as an intersection}
We have $ \mathcal{HG}_T (M_L; \Sigma)  \cong \mathcal{G} (M_L ; \Sigma) \cap \mathcal{H} (\Sigma)$, 
where $\mathcal{H} (\Sigma)$ is the hyperelliptic mapping class group with respect to 
the hyperelliptic involution $T|_{\Sigma}$.  
\end{thm}
In Theorem \ref{introthm:hyperelliptic Goeritz group as an intersection} 
we regard both $\mathcal{G} (M_L ; \Sigma)$ and $\mathcal{H} (\Sigma)$ as subgroups of 
$ \MCG  (\Sigma)$. 

Using the two isomorphisms in Theorems \ref{introthm:HG(M;V)/i} and 
\ref{introthm:hyperelliptic Goeritz group as an intersection}, 
we can obtain finite presentations for the Goeritz groups of 
$3$-bridge decompositions of $2$-bridge links, which is highly non-trivial, 
see Example \ref{ex:Goeritz group of some 3-bridge decomposition}. 

For a Heegaard splitting $M = V^+ \cup_\Sigma V^-$, 
Hempel \cite{Hempel01} introduced a measure of complexity 
called the {\it distance}. 
Roughly speaking, this is defined to be the distance between the sets of 
the boundaries of meridian disks of $V^+$ and 
$V^-$ in the curve graph $\mathcal{C} (\Sigma)$. 
The distance gives a nice way to describe the structure of the Goeritz groups 
of Heegaard splittings. 
In fact, by work of Namazi \cite{Namazi07} and Johnson \cite{Johnson10}, 
it turned out that the Goeritz group is always a finite group if the distance of the Heegaard splitting is at least $4$. 
This is completely different from the case for the splittings of distance at most $1$, 
where the Goeritz group is always an infinite group (see Johnson-Rubinstein \cite{JohnsonRubinstein13} and 
Namazi \cite{Namazi07}). 
The notion of distance can naturally be defined for bridge decompositions as well, 
see for example Bachman-Schleimer \cite{BachmanSchleimer05}. 
It is hence quite natural to expect that the Goeritz groups of bridge decompositions 
hold similar properties as above. 
In Example \ref{ex:distance and infinite order Goeritz groups} we actually see that 
the Goeritz group is always an infinite group 
when the distance of the bridge decomposition is at most $1$ 
(except the case of the $2$-bridge decomposition of the trivial knot), 
and in Section \ref{section_distance}, we prove the following. 
\begin{thm}
\label{introthm:finite Goeritz group} 
There exists a uniform constant $N>0$ such that 
for each integer $n$ at least $3$, 
if the distance of an $n$-bridge decomposition 
$(L; S)$ of a link $L$ in $S^3$
is greater than $N$, 
then $\mathcal{G}(L;S)$ is a finite group.   
\end{thm}
We will show that the constant $N$ in the above theorem can actually be taken 
to be at most $3796$. This number is, however, far bigger than 
the constant $4$ in the case of Heegaard splittings. 
The condition ``$n \geq 3$" in the above theorem is crucial. 
In fact, we will see in Example \ref{ex:Goeritz group of 2-bridge decomposition} that 
the Goeritz group of a $2$-bridge  decomposition 
is always a finite group except for the case of the trivial 2-component link.

In Sections \ref{section_stabilized} and \ref{section_application}, 
we will discuss the Goeritz groups of bridge decompositions 
of links whose orders are infinite. 
Let $\Sigma_{g, m}$ denote the orientable surface of genus $g$ with $m$ marked points, possibly $m=0$. 
By $\Sigma_g$ we mean $\Sigma_{g,0}$ for simplicity. 
The mapping class group $ \MCG (\Sigma_{g, m})$  is 
the group of isotopy classes of orientation-preserving 
self-homeomorphisms of $\Sigma_{g, m}$ which preserve the marked points setwise. 
Assume $3g-3+ m \ge 1$. 
By Nielsen-Thurston classification, 
elements  in $ \MCG (\Sigma_{g, m})$ fall into three types: 
periodic, reducible, pseudo-Anosov \cite{Thurston88,FarbMargalit12,FLP}. 
This classification can also be applied for 
elements of the Goeritz group $\mathcal{G}(L; S)$ of {an $n$-bridge} decomposition $(L; S)$ 
by regarding $\mathcal{G}(L; S)$ as a subgroup of $\MCG (S , S \cap L) = \MCG (\Sigma_{0, 2n})$.

Given a bridge decomposition $(L; S)$ of a link $L \subset S^3$,  
consider the bridge decomposition $(L; S_{(p,1)})$ obtained by the 
{\it $1$-fold stabilization} of $(L; S)$ at a point $p \in S \cap L$ 
(see Section \ref{subsection_Bridge-decompositions} for the 
definition of a stabilization).
This bridge decomposition has distance at most $1$ 
(see Lemma \ref{lem: distance stabilized}), and hence the order of its Goeritz group is infinite 
(except the case of the $2$-bridge decomposition of the trivial knot, 
which is the $1$-fold stabilization of the $1$-bridge decomposition of the 
trivial knot).  
Therefore, it is quite natural to ask whether it contains a pseudo-Anosov element. 
In Section \ref{section_stabilized} we give a complete answer to this question as in the following way. 
\begin{thm}
\label{introthm_pA elements in a stabilized bridge decomposition}
Let $(L; S)$ be an $n$-bridge decomposition of a link $L$ in $S^3$ 
with $n \geq 2$. 
Let $p$ be an arbitrary point in $S \cap L$. 
If $(L; S)$ is the $2$-bridge decomposition of the $2$-component trivial link, 
then $\mathcal{G} (L; S_{(p,1)})$ is an infinite group 
consisting only of reducible elements. 
Otherwise, the Goeritz group $\mathcal{G} (L; S_{(p,1)})$ contains a pseudo-Anosov element. 
\end{thm}



Consider the sequence 
$$(L; S_{(p,1)}), (L; S_{(p,2)}), \ldots , (L; S_{(p,k)}) , \ldots $$
obtained by a $k$-fold stabilization of a bridge decomposition $(L; S)$ 
at a point $p \in S \cap L$ for each positive integer $k$. 
By Theorem \ref{introthm_pA elements in a stabilized bridge decomposition}, 
the Goeritz groups of the bridge decompositions in this sequence 
always have pseudo-Anosov elements. 
We are interested in how big those Goeritz groups can be from the 
view point of dynamical properties of their elements.  
In the final section of the paper we actually discuss 
the asymptotic behavior of the minimal pseudo-Anosov dilatations in the Goeritz groups with respect to 
stabilizations. 
To each pseudo-Anosov element $\phi \in \MCG (\Sigma_{g,m}) $, 
there is an associated dilatation (stretch factor) $\lambda(\phi)>1$. 
The logarithm 
$\log \lambda(\phi)$ of the dilatation is called the {\it entropy} of $\phi$. 

Fix a surface $\Sigma_{g, m}$ and consider 
the set of entropies 
$$\{\log \lambda(\phi)\ |\ \phi \in  \MCG (\Sigma_{g, m})\  \mbox{is pseudo-Anosov}\}.$$ 
This is a closed, discrete subset of ${\Bbb R}$ due to Ivanov \cite{Ivanov88}. 
For any subgroup $G \subset  \MCG (\Sigma_{g, m})$ containing a pseudo-Anosov element,  
let $\ell(G)$ denote the minimum of entropies 
over all pseudo-Anosov elements in $G$. 
Note that $\ell(G) \ge \ell( \MCG (\Sigma_{g, m}))$. 
Penner~\cite{Penner91} proved that 
$\ell( \MCG (\Sigma_g)) $ is comparable to $1/g$, and 
Hironaka-Kin \cite{HironakaKin06} proved that $\ell ( \MCG (\Sigma_{0, m}) )$ is 
comparable to $\frac{1}{m}$. 
Here, for real valued functions $f, h: X \rightarrow {\Bbb R}$ on a subset $X \subset {\Bbb N}$,  
we say that $f$ {\it is comparable to} $h$, and write  $f \asymp h$,  
if there exists a constant $c>0$  such that $h(x)/c \le f(x) \le c h(x)$ for all $x \in X$.  
For the genus-$g$ handlebody group $\MCG(V_g) \subset \MCG (\Sigma_g)$, 
Hironaka \cite{Hironaka14} proved that $\ell (\MCG (V_g))$ is also comparable to $\frac{1}{g}. $




In Section \ref{section_application}, we show the following.

\begin{thm}
\label{introthm_asymptitic behavior for the trivial knot}
Let $(O; n)$ be the $n$-bridge decomposition of the trivial knot $O \subset S^3$. 
Then we have 
$$\ell (\mathcal{G}(O; n))  \asymp \dfrac{1}{n}.$$
\end{thm}

We note that the sequence $(O; 2) , (O; 3) , \ldots$ we consider in the above theorem is nothing but 
that of finite fold stabilizations of the $1$-bridge decomposition $(O; 1)$ 
of the trivial knot $O$ (at any point).

\begin{thm}
\label{introthm_asymptotic behavior for the Hopf link}
Let $(H; S_{(p, n-2)})$ be the $n$-bridge decomposition of the Hopf link 
$H \subset S^3$ obtained from 
the $2$-bridge decomposition $(H; S)$ by the $(n-2)$-fold stabilization at a point 
$p \in S \cap H$. 
Then we have 
$$\ell(\mathcal{G}(H; S_{(p, n-2)})) \asymp \dfrac{1}{n}.$$
\end{thm}
The proofs of Theorems 
\ref{introthm_asymptitic behavior for the trivial knot}
and \ref{introthm_asymptotic behavior for the Hopf link}
are based on the braid-theoretic 
description (Theorem \ref{thm_bridge_char}) of the Goeritz groups of bridge decompositions 
derived from the earlier-mentioned identification 
$\mathcal{G} (L; S) =  \MCG  (B^+, B^+ \cap L) \cap  \MCG  (B^-, B^- \cap L)$ for a bridge decomposition 
$(L; S)$. 
Theorems 
\ref{introthm_asymptitic behavior for the trivial knot} 
and 
\ref{introthm_asymptotic behavior for the Hopf link}, 
together with 
Theorems \ref{introthm:HG(M;V)/i} and 
\ref{introthm:hyperelliptic Goeritz group as an intersection}, 
allow us to have the following intriguing application to 
the Goeritz groups of Heegaard splittings. 

\begin{cor}
\label{introcor_sphere-entropy}
Let $(S^3; g)$ be the genus-$g$ Heegaard splitting of $S^3$ for $g \geq 0$. 
 Then we have 
$$\ell(\mathcal{G}(S^{3}; g)) \asymp \dfrac{1}{g}.$$
\end{cor}

\begin{cor}
\label{introcor_rp3-entropy}
Let $(\mathbb{RP}^3 ; g)$ be the genus-$g$ Heegaard splitting of the real projective space 
$\mathbb{RP}^3$ for $g \geq 1$. 
Then we have 
 $$\ell(\mathcal{G}(\mathbb{RP}^{3}; g)) \asymp \dfrac{1}{g}.$$ 
\end{cor}


\vspace{1em}

Throughout the paper, we will work in the piecewise linear category. 
For a subspace $Y$ of a space $X$,  
$N(Y; X)$ denotes a regular neighborhood of $Y$ in $X$, 
$\Cl(Y; X)$ the closure of $Y$ in $X$,  and 
$\Int (Y)$ the interior of $Y$. 
When $X$ is a metric space, 
$N_\varepsilon (Y; X)$ denotes the closed $\varepsilon$-neighborhood of $Y$. 
If there is no ambiguity about the ambient space in question, 
the $X$ is suppressed from the notation. 
The number of components of $X$ is denoted by $\# X$. 

\section{Preliminaries} 
\label{section_Preliminaries}

Let $X_1, \ldots, X_n$ be possibly empty subspaces 
of an orientable manifold $M$. 
Let $\Homeo_+ (M, X_1, \ldots, X_n)$ denote 
the group of orientation-preserving self-homeomorphisms of $M$ 
that map $X_i$ onto $X_i$ for each $i=1 , \ldots , n$. 
The {\it mapping class group}, denoted by $\MCG  (M , X_1, \ldots, X_n)$,  
is defined by 
$$ \MCG  (M , X_1,  \ldots, X_n) = \pi_0 (\Homeo_+ (M, X_1, \ldots, X_n)). $$
Here, we do not require that the maps and isotopies fix the points 
in $\partial M$. 
For a compact orientable surface $\Sigma$ with marked points, 
by $\MCG (\Sigma)$ we mean $\MCG  (\Sigma, \{ p_1, \ldots, p_m \})$, where 
$\{ p_1, \ldots, p_m \}$ is the set of marked points of $\Sigma$.  
We apply elements of mapping class groups from right to left, i.e., 
the product  $fg$ means that $g$ is applied first.

\subsection{Heegaard splittings}
\label{subsection_Heegaard-splittings}

A {\it handlebody} of genus $g$ is an oriented 3-manifold 
obtained from a 3-ball by attaching $g$ copies of a 1-handle. 
Every closed orientable 3-manifold $M$ can be obtained by gluing together two handlebodies 
$V^+$ and $V^-$ of the same genus $g$ for some $g \ge 0$, that is, 
$M$ can be represented as $M = V^+ \cup V^-$ and 
$V^+ \cap V^- = \partial V^+ = \partial V^- = \Sigma \cong \Sigma_g$. 
We denote such a decomposition by $(M; \Sigma)$ or $V^+ \cup_{\Sigma} V^-$, 
and we call it a {\it genus-$g$ Heegaard splitting}  of  $M$. 
The surface $ \Sigma$ is called the {\it Heegaard surface} of the splitting. 
We say that two Heegaard splittings 
$(M; \Sigma)$ and $(M; \Sigma')$ of $M$ are {\it equivalent} 
if the Heegaard surfaces $\Sigma$ and $\Sigma'$ are isotopic in $M$. 

We recall the notion of stabilization 
for a  Heegaard splitting $V^+ \cup_{\Sigma} V^-$ of $M$. 
Take a properly embedded arc $\gamma$ in $V^+$ which is parallel to $\Sigma$ (Figure \ref{fig_Hstab}). 
We denote the union $V^- \cup N(\gamma)$ by $\widehat{V}^-$, 
the closure of $V^+ \setminus N(\gamma)$ by $\widehat{V}^+$, and 
their common boundary $\partial \widehat{V}^+= \partial \widehat{V}^-$ by $\widehat{\Sigma}$. 
Then $ \widehat{V}^+ \cup_{\widehat{\Sigma}}  \widehat{V}^-$ 
is again a Heegaard splitting of $M$, 
where the genus of $\widehat{\Sigma}$ is one greater than that of $\Sigma$. 
Note that $\widehat{\Sigma}$ does not depend on $\gamma$. 
We say that $\widehat{\Sigma}$ {\it is obtained from} $\Sigma$ {\it by a stabilization}.  

\begin{center}
\begin{figure}[t]
\includegraphics[height=1.8cm]{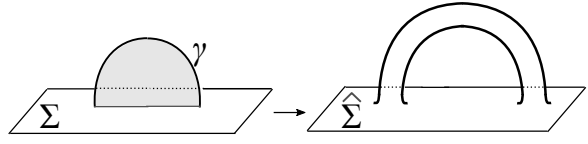}
\caption{
A stabilization for a Heegaard surface.}
\label{fig_Hstab}
\end{figure}
\end{center}

By Waldhausen \cite{Waldhausen68}, the 3-sphere $S^3$ admits 
a unique genus-$g$ Heegaard splitting up to equivalence for each $g \geq 0$. 
Similarly, by Bonahon-Otal \cite{BonahonOtal83}, 
a lens space admits a unique 
genus-$g$ Heegaard splitting up to equivalence for each $g \geq 1$.

\subsection{Goeritz groups of Heegaard splittings}
\label{subsection_GH}

Let $V= V_g$ be a handlebody of genus $g$. 
We call $\mathrm{MCG}  (V)$ a {\it handlebody group}. 
Since the map 
$$\MCG  (V) \rightarrow \MCG  (\partial V)$$ 
sending $[f] \in \MCG  (V)$ to $[f|_{\partial V}] \in \MCG  (\partial V)$ is injective, 
we regard $\MCG  (V) $ as a subgroup of $\MCG  (\partial V)$. 

Suppose that $M$ admits a genus-$g$ Heegaard splitting 
$(M; \Sigma) = V^+ \cup_{\Sigma} V^- $. 
We equip the common boundary $\partial V^+ = \partial V^-(= \Sigma)$ with 
the orientation induced  by that of $V^-$. 
The {\it Goeritz group}, denoted by $\mathcal{G}(M; \Sigma)$ or $\mathcal{G}(V^+ \cup_{\Sigma} V^-)$, 
of the Heegaard splitting is defined by 
$$ \mathcal{G}(M; \Sigma) = \MCG  (M, V^+). $$
We note that $\MCG  (M, V^+) = \MCG  (M, V^-)$. 
We can regard $\mathcal{G}(M; \Sigma) $ as a subgroup of both $\MCG  (V^+)$ and $\MCG  (V^-)$. 
Further, regarding $\MCG  (V^+)$ and $\MCG  (V^-)$ as subgroups of $\MCG  (\Sigma)$, 
the group $\mathcal{G}(M; \Sigma)$ is nothing but the intersection 
$\MCG  (V^+) \cap \MCG  (V^-)$. 

When $(M; \Sigma)$ is a unique genus-$g$ Heegaard splitting of $M$ up to equivalence, 
we simply call $\mathcal{G}(M; \Sigma) $ the {\it genus-$g$ Goeritz group} of $M$, 
and we denote it by $\mathcal{G}(M; g)$.

\subsection{Bridge decompositions}
\label{subsection_Bridge-decompositions}

Let $\mathcal{T}=  T_{1} \cup \cdots \cup T_{n}$ be $n$ disjoint arcs properly embedded in 
a 3-ball $B$. 
We call $\mathcal{T}$ an {\it $n$-tangle} or simply a  {\it tangle}. 
The endpoints of $\mathcal{T}$, denoted by $\partial \mathcal{T}$, mean the set 
$\partial T_1 \cup \cdots \cup \partial T_n$ 
of $2n$ points in $\partial B$. 

Suppose that $\mathcal{T}$ and $\mathcal{T}'$ are $n$-tangles in $B$ 
such that they share the same endpoints,  that is, $\partial \mathcal{T}= \partial \mathcal{T}'$. 
We say that $\mathcal{T}$ and $\mathcal{T}'$  are {\it equivalent}  
if there exists an orientation-preserving homeomorphism $f: B \rightarrow B$ 
sending $\mathcal{T}$ to $\mathcal{T}'$ with  $f|_{\partial B}=  \id_{\partial B}$.  
In this case, we write $\mathcal{T}= \mathcal{T}'$. 
 
In what follows, when we consider an $n$-tangle in the 3-ball $B$, 
we always adopt the following convention. 
\begin{itemize}
\item
We identify the boundary $\partial B$ of the $3$-ball $B$ with $S^2$, and 
we  implicitly fix an oriented great circle $C$ of $\partial B$. 
\item
We fix $2n$ points labeled $ 1, 2, \ldots, 2n $ on $C$, where the order of the labels is 
compatible with the prefixed orientation of $C$.  
\item
The endpoints $\partial \mathcal{T}$ of a given $n$-tangle $\mathcal{T}$ is exactly 
the set $\{ 1, 2, \ldots, 2n \}$. 
\item
When we show a figure of a given $n$-tangle $\mathcal{T}$, 
we draw $C$ as a horizontal line 
in such a way that the labeled points $ 1, 2, \ldots, 2n $ are ordered 
from left to right as in Figure \ref{fig_BandC}. 
The sphere $S := \partial B$ near $C$ is perpendicular to the paper plane in such a figure.

\end{itemize}
This convention will be particularly important in the arguments in Section \ref{subsection_wicket-groups}.

Let $\mathcal{A}= \mathcal{A}_n $ be the $n$-tangle in $B$  
as in Figure~\ref{fig_tangleA}.
We say that an $n$-tangle in $B$ is {\it standard} 
if it is equivalent to $\mathcal{A}$. 

\begin{center}
\begin{figure}[t]
\includegraphics[height=2cm]{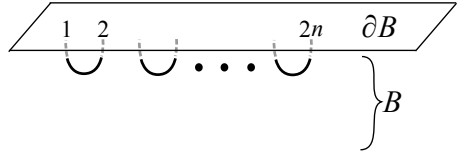}
\caption{The standard $n$-tangle  $\mathcal{A}= \mathcal{A}_n$.} 
\label{fig_tangleA}
\end{figure}
\end{center}

We say that an $n$-tangle $\mathcal{T}$ in $B$ is {\it trivial} 
if there exists an orientation-preserving homeomorphism $f: B \rightarrow B$ 
sending  $\mathcal{A}$ to $\mathcal{T}$. 
Here, $f$ does not necessarily respect the label of the endpoints. 
In other words, an $n$-tangle $\mathcal{T}$ in $B$ is trivial 
if we can move $N (\partial \mathcal{T} ; \mathcal{T})$  
by an isotopy keeping 
$\partial \mathcal{T}$ lying in $\partial B$ 
so that the resulting tangle is equivalnt to $\mathcal{A}$. 
For example, Figure \ref{fig_BandC} shows examples of trivial $n$-tangles 
(1) $\mathcal{B} = \mathcal{B}_n$ and (2) $\mathcal{C}= \mathcal{C}_n$. 
\begin{center}
\begin{figure}[t]
\includegraphics[height=2.4cm]{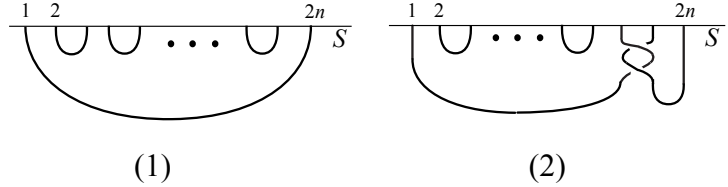}
\caption{The $n$-tangles (1) $\mathcal{B}= \mathcal{B}_n$  and 
(2) $\mathcal{C} =  \mathcal{C}_n$.} 
\label{fig_BandC}
\end{figure}
\end{center}

Consider the genus-$0$  Heegaard splitting $B^+ \cup_S B^-$ of $S^3$, 
where $B^+$ and $B^-$ are $3$-balls and 
$S= \partial B^+ = \partial B^-$. 
Let $\mathbb{R}^{3}_{+} = \{ (x,y,z) \in \mathbb{R}^{3}  \, | \, z \geq 0\}$ and 
$\mathbb{R}^{3}_{-} = \{ (x,y,z) \in \mathbb{R}^{3} \, | \, z \leq 0\}$. 
We identify 
$B^+$ with  $\mathbb{R}^{3}_{+}\cup \{ \infty \}$, 
$B^-$ with $ \mathbb{R}^{3}_{-} \cup \{ \infty \}$. 
Then  $S$ is written by $S= \{ (x,y,z) \in \mathbb{R}^{3} \, | \, z = 0\} \cup \{\infty\}$. 
Define an involution  $\rho : S^3 \to S^3$ by 
$\rho (x,y,z) = (x,y,-z)$. 
Note that $\rho |_{S} =  \id_S$, and $\rho$ interchanges $B^+$ with $B^-$. 
This means that $\rho$ interchanges $n$-tangles  in $B^+$ with $n$-tangles  in $B^-$. 
Consider the standard $n$-tangle $\mathcal{A}= \mathcal{A}_n$ in $B^-$. 
We set 
$\Bar{\mathcal{A}}= \Bar{\mathcal{A}_n}   = \rho(\mathcal{A}_n)$, 
which is a trivial $n$-tangle in $B^+$. 
Hereafter, 
we illustrate the splitting $B^+ \cup_S B^-$ as in Figure \ref{fig_rho}: 
again, the horizontal line indicates the sphere $S$, 
the $3$-ball $B^-$ lies below $S$ and 
the other $3$-ball $B^+$ lies above $S$. 

\begin{center}
\begin{figure}[t]
\includegraphics[height=3.2cm]{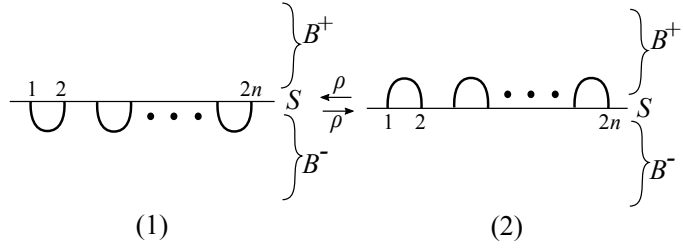}
\caption{
The involution $\rho: S^3 \rightarrow S^3$ interchanges 
(1) the $n$-tangle  $\mathcal{A}$ in $B^{-}$ with 
(2) the $n$-tangle  $\Bar{\mathcal{A}}$ in $B^{+}$.}
\label{fig_rho}
\end{figure}
\end{center}

Let $L$ be a link, possibly a knot  in $S^3 $. 
Suppose that 
$B^+ \cap L$ and $B^- \cap L$ are trivial $n$-tangles. 
Then we have the decomposition  
$$ (S^3, L) = (B^+, B^+ \cap L  )  \cup (B^-, B^- \cap L ) ,$$ 
which is denoted by $(L; S)$ or $(B^+ \cap L) \cup_S (B^- \cap L)$. 
We call such a decomposition  an {\it $n$-bridge decomposition} of $L$. 
We also call $S = \partial B^+ = \partial B^-$ a {\it bridge sphere} of $L$.

We say that  two $n$-bridge decompositions 
$(L; S)$ and $(L; S')$
are {\it equivalent}  
if $S$ and $S'$ are isotopic through  bridge spheres of $L$. 

A stabilization of an $n$-bridge decomposition 
$(L; S)= (B^+ \cap L) \cup_S (B^- \cap L)$ 
is defined as follows. 
Take a point $p \in L \cap S$. 
We deform the bridge sphere $S$ 
near $p$  into a sphere $S_{(p,k)}$ so that 
the cardinality  of the intersection $L \cap S_{(p,k)}$ 
increases by $2k$ as illustrated 
in Figure \ref{fig_sphere-stab}(2). 
 More precisely, 
 let $U$ be a disk embedded in $B^+$ 
 whose boundary consists of three arcs $\alpha$, $\beta$ and $\gamma$, 
 where $\alpha= U \cap L$, $\beta = U \cap S$, see Figure \ref{fig_sphere-stab}(1). 
 Then $\gamma \cap L$ consists of an endpoint of $\gamma$, 
 and $\gamma \cap S$ consists of the other endpoint of $\gamma$. 
 Let $N(\gamma)$ be a regular neighborhood of $\gamma$. 
 We denote the union $B^- \cup N(\gamma)$ by $B^-_{(p,1)}$, 
 the closure of $B^+ \setminus N(\gamma)$ by $B^+_{(p,1)}$, 
 their common boundary by $S_{(p,1)}$. 
For $k \ge 1$, take $k$ parallel copies 
$\gamma'_1, \ldots , \gamma'_k$ of $\gamma$, and consider the union 
 $\gamma_k = \gamma'_1 \cup \cdots \cup \gamma_k' \subset U$. 
 Let $N(\gamma_k)$ be a closed regular neighborhood of $\gamma_k$. 
 We denote the union $B^- \cup N(\gamma_k)$ by $B^-_{(p,k)}$, 
 the closure of $B^+ \setminus N(\gamma_k)$ by $B^+_{(p,k)}$, 
 their common boundary by $S_{(p,k)}$. 
Then  
$$(L; S_{(p,k)}) = (B^+_{(p,k)} \cap L) \cup_{S_{(p,k)}} (B^-_{(p,k)} \cap L)$$
is an $(n+k)$-bridge decomposition. 
Note that $S_{(p,k)}$ does not depend on the disk $U$. 
It only depends on $S$, $p$ and $k$. 
We say that  $S_{(p,k)}$ 
{\it is obtained from} $S$ {\it by a  $k$-fold stabilization} ({\it at $p$}). 
When $L$ is a knot, the stabilized bridge decomposition 
$(L; S_{(p,k)})$ does not depend on the choice of the point $p$ 
in $L \cap S$. 
See Jang-Kobayashi-Ozawa-Takao 
\cite{JangKobayashiOzawaTakao19} for a rigorous proof of this fact.

\begin{center}
\begin{figure}[t]
\includegraphics[height=3cm]{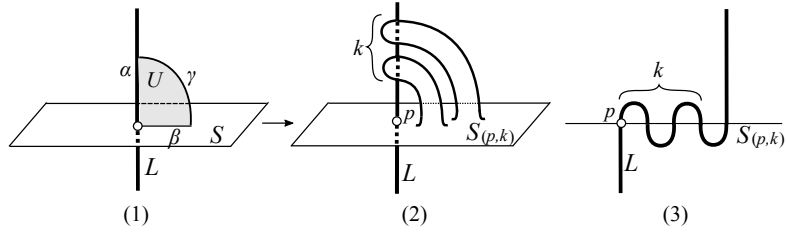}
\caption{A stabilization for a bridge sphere. 
(1) A bridge sphere $S$. 
(2) The bridge sphere $S_{(p,k)}$ obtained from $S$ by a $k$-fold stabilization, where $k=2$. 
(3) Illustration of the bridge decomposition $(L; S_{(p,k)})$.}
\label{fig_sphere-stab}
\end{figure}
\end{center}

It is proved by Otal \cite{Otal82} that 
for each $n \ge 1$, an $n$-bridge decomposition of the trivial knot $O$ is unique up to equivalence. 
We denote the $n$-bridge decomposition of $O$ by $(O; n)$. 
We note that the same consequence holds as well 
for the $2$-bridge knots by Otal \cite{Otal85}, and 
the torus knots by  Ozawa \cite{Ozawa11}.

Consider the $1$-bridge decomposition $(O; 1)$ with the bridge sphere $S$. 
Let $p $ be a point in $O \cap S$, and 
let $S_{(p,n-1)}$ be the bridge sphere obtained from $S$ by an $(n-1)$-fold stabilization, 
see Figure \ref{fig_trivial-bridge}(1). 
The resulting $n$-bridge decomposition of $O$ can be expressed 
by using the trivial tangles $\Bar{\mathcal{A}} = \Bar{\mathcal{A}}_n $ and 
$\mathcal{B}= \mathcal{B}_n$ as follows. 
$$(O; n)= \Bar{\mathcal{A}} 
\cup_{S_{(p, n-1)}} \mathcal{B}.$$ 

\begin{center}
\begin{figure}[t]
\includegraphics[height=3cm]{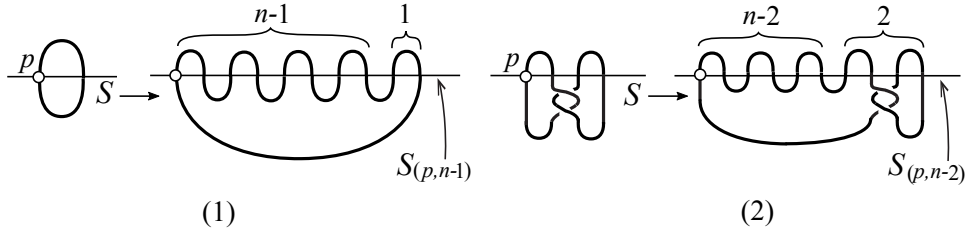}
\caption{(1) The $n$-bridge sphere 
of the trivial knot $O$ obtained from the stabilization of the $1$-bridge sphere.
(2) The $n$-bridge sphere 
of the Hopf link $H$ obtained from the stabilization of the $2$-bridge sphere.}
\label{fig_trivial-bridge}
\end{figure}
\end{center}

For the $2$-bridge decomposition $(H; S)$ of the Hopf link $H$, 
we pick a point $p \in H \cap S$ as in Figure \ref{fig_trivial-bridge}(2). 
 Let $S_{(p, n-2 )}$ be the bridge sphere obtained from $S$ by an $(n-2)$-fold stabilization. 
The resulting  $n$-bridge decomposition of $H$ is of the form 
$$(H; S_{(p, n-2)})= \Bar{\mathcal{A}} 
\cup_{S_{( p, n-2)}} \mathcal{C}$$
by using the trivial tangles $\Bar{\mathcal{A}}=\Bar{\mathcal{A}}_n $ and 
$\mathcal{C}=\mathcal{C}_n $.

\subsection{Curve graphs}
\label{subsec:Curve graphs}

Let $\Sigma$ be a compact orientable surface. 
A simple closed curve on $\Sigma$ is said to be {\it essential} 
if it does not bound a disk in $\Sigma$ and it is not parallel to a component of the boundary.  
A properly embedded arc in $\Sigma$ is said to be {\it essential} if it cannot be isotoped (rel. $\partial \Sigma$) into $\partial \Sigma$. 

Suppose that $\Sigma$ is not an annulus. 
The {\it curve graph} $\mathcal{C}(\Sigma)$ of $\Sigma$ is defined to be 
the $1$-dimensional simplicial complex 
whose vertices are the isotopy classes of essential simple closed curves on $\Sigma$ 
and a pair of distinct vertices spans an edge if and only if 
they admit disjoint representatives. 
By definition, when $\Sigma$ is a torus, a $1$-holed torus or a 
$4$-holed sphere, 
$\mathcal{C}(\Sigma)$ has no edges. 
In these cases, we alter the definition slightly for convenience. 
When $\Sigma$ is a torus or a 1-holed torus, 
two distinct vertices of $\mathcal{C}(\Sigma)$ span an edge 
if and only if their geometric intersection number is equal to $1$. 
When $\Sigma$ is a $4$-holed sphere, 
two vertices of $\mathcal{C}(\Sigma)$ span an edge 
if and only if their geometric intersection number is equal to $2$. 

Similarly, 
the {\it arc and curve graph} $\mathcal{AC}(\Sigma)$ of $\Sigma$ is the $1$-dimenaional simplicial complex defined as follows.
When $\Sigma$ is not an annulus, 
the vertices of $\mathcal{AC}(\Sigma)$ are the isotopy classes of essential simple closed curves and 
isotopy classes of essential arcs (rel. $\partial \Sigma$) on $\Sigma$. 
A pair of distinct vertices spans an edge if and only if 
they admit disjoint representatives. 
When $\Sigma$ is an annulus, 
the vertices of $\mathcal{AC}(\Sigma)$ are isotopy classes of essential arcs (rel. endpoints).  
Two distinct vertices spans an edge if and only if 
they admits disjoint representatives. 
In this case, we set $\mathcal{C}(\Sigma) := \mathcal{AC}(\Sigma)$ for convenience.  

By $\mathcal{C}^{(0)}(\Sigma)$ and $\mathcal{AC}^{(0)}(\Sigma)$ 
we denote the set of vertices of $\mathcal{C}(\Sigma)$ and $\mathcal{AC}(\Sigma)$,  respectively. 
We can regard $\mathcal{C}(\Sigma)$ (respectively, $\mathcal{AC}(\Sigma)$)  as the geodesic metric space 
equipped with the simplicial metric $d_{\mathcal{C}(\Sigma)}$ (respectively, $d_{\mathcal{AC}(\Sigma)}$). 

Let $\delta>0$. 
A geodesic metric space is said to be {\it $\delta$-hyperbolic} 
if any geodesic triangle is {\it $\delta$-slim}, 
that is, 
each side of the triangle lies in the closed $\delta$-neighborhood of the union of the other two sides. 
 
Recent independent work by  
Aougab \cite{Aougab13}, Bowditch \cite{Bowditch14}, 
Clay-Rafi-Schleimer \cite{ClayRafiSchleimer14} and 
Hensel-Przytycki-Webb \cite{HenselPrzytyckiWebb15} 
after a famous work on the hyperbolicity of $\mathcal{C}(\Sigma)$ 
by Masur-Minsky \cite{MasurMinsky99}  shows the following.

\begin{thm}
\label{thm:hyperbolicity}
The curve graph $\mathcal{C}(\Sigma)$ is a $\delta$-hyperbolic space, where 
$\delta$ does not depend on the topological type of $\Sigma$.
\end{thm}
We note that in \cite{HenselPrzytyckiWebb15} it was shown that the constant $102$
\footnote{
In \cite{HenselPrzytyckiWebb15} the hyperbolicity constant is 
defined 
using the {\it $k$-centered triangle} condition instead of 
the $\delta$-slim triangle 
condition, which we adopt in this paper. 
The claim of \cite{HenselPrzytyckiWebb15} is that any geodesic triangle of $\mathcal{C} (\Sigma)$ is 
$17$-centered.  
By Bowditch \cite[Lemma 6.5]{Bowditch06}, this implies that 
$\mathcal{C} (\Sigma)$ is $17 \cdot 6 = 102$-hyperbolic.} is enough 
for the hyperbolicity constant $\delta$ in the above theorem.

Let $\Sigma$ be a compact orientable surface with a negative Euler characteristic. 
Recall that a subsurface $Y$ of $\Sigma$ is said to be 
{\it essential} if each component of $\partial Y$ is not contractible in $\Sigma$. 
We do not allow annuli homotopic to a component of $\partial \Sigma$ to be essential subsurfaces. 
We always assume that essential subsurfaces are connected and proper. 
Let $Y$ be an essential subsurface of $\Sigma$. 
The {\it subsurface projection} $\pi_{Y}:\mathcal{C}^{(0)}(\Sigma) \rightarrow \mathcal{P}(\mathcal{C}^{(0)}(Y))$, 
where $\mathcal{P} (~ \cdot ~)$ denotes the power set, 
is defined as follows. 
First, we consider the case where $Y$ is not an annulus. 
Define $\kappa_{Y}:  \mathcal{C}^{(0)}(\Sigma) \rightarrow \mathcal{P}(\mathcal{AC}^{(0)}(Y))$ to be the map 
that takes $\alpha \in \mathcal{C}^{(0)}(\Sigma)$ to $\alpha \cap Y \subset \mathcal{P}(\mathcal{AC}^{(0)}(Y))$. 
Further, define $\sigma_{Y}:\mathcal{AC}^{(0)}(Y) \rightarrow \mathcal{P}(\mathcal{C}^{(0)}(Y))$ 
by taking $\alpha \in \mathcal{AC}^{(0)}(Y)$ to the set of essential simple closed curves that  are components of the boundary of $N (\alpha \cup \partial Y;Y)$. 
The map $\sigma_{Y}$ is naturally extends to the map 
$\sigma_{Y}:\mathcal{P}(\mathcal{AC}^{(0)}(Y)) \rightarrow \mathcal{P}(\mathcal{C}^{(0)}(Y))$. 
The map $\pi_Y$ is then defined by $\pi_{Y}:=  \sigma_{Y} \circ \kappa_{Y} $. 
Next we consider the case where $Y$ is an annulus. 
Fix a hyperbolic metric on $\Sigma$. 
Let $p:\tilde{Y} \rightarrow \Sigma$ be the covering map 
corresponding to $\pi_{1}(Y)$. 
Let $\hat{Y}$ be the metric completion of $\tilde{Y}$. 
We can identify $Y$ with $\hat{Y}$. 
Suppose that $\alpha \in \mathcal{C}^{0}(\Sigma)$. 
We can regard $p^{-1}(\alpha)$ as the set of properly embedded arcs in $Y$ 
and define $\pi_{Y}(\alpha)$ to be the set of the properly embedded arcs that are essential in $Y$. 
The following theorem, called the {\it bounded geodesic image theorem}, was proved by 
Masur-Minsky \cite{MasurMinsky00}. 
\begin{thm}
\label{thm:geodesic image}
Let $\Sigma$ be a compact orientable surface with a negative Euler characteristic. 
Then there exists a constant $C>0$ satisfying the following condition. 
Let $Y \subsetneq \Sigma$ be an essential subsurface that is not a $3$-holed sphere.  
Let $c$ be a geodesic in $\mathcal{C}(\Sigma)$ such that  
$\pi_{Y}(\alpha)\not=\emptyset$ for any vertex $\alpha$ of $c$. 
Then it holds $\mathrm{diam}_{\mathcal{C}(Y)}(\pi_{Y}(c)) \le C$.  
\end{thm}

We remark that 
Webb \cite {Webb15} showed that the constant $C$ in the above theorem 
can be taken to be independent of the topological type of $\Sigma$.

\subsection{The distance of bridge decompositions}
\label{subsec:The distance of bridge decompositions}
Let $L$ be a link in $S^{3}$, and 
let $(B^{+} \cap L) \cup_{S} (B^{-} \cap L)$ be an $n$-bridge decomposition of $L$ 
with $n \geq 2$. 
Set $S_L:= \Cl(S - N (S \cap L; S))$. 
We denote by $\mathcal{D}^{+}$ (respectively, $\mathcal{D}^{-}$) the set of 
vertices of $\mathcal{C}(S_{L})$ 
that are represented by simple closed curves bounding 
disks in $B^{+}-L$ (respectively $B^{-}-L$). 
The {\it distance} $d(L;S)$ of the bridge decomposition $(S; L)$ 
is defined by 
$d(L;S):=\mathrm{min}\,d_{\mathcal{C}(S_{L})}(\alpha,\beta)$, 
where the minimum is taken over all $\alpha \in \mathcal{D}^{+}$ and $\beta \in \mathcal{D}^{-}$.

\begin{lem}
\label{lem: distance stabilized}
The distance of a stabilized bridge decomposition of a link in $S^3$ is at most $1$. 
\end{lem}
\begin{proof}
Let $(L; S) = (B^+ \cap L) \cup_S (B^- \cap L)$ be an $n$-bridge decomposition of $L$. 
Take a point $p \in S \cap L$. 
Consider the $(n+1)$-bridge decomposition 
$$(L; S_{(p, 1)}) = (B^+_{(p, 1)} \cap L) \cup_{S_{(p,1)}} (B^-_{(p,1)} \cap L).$$  
It suffices to show that $d (L; S_{(p,1)})$ is at most $1$. 
If $n=1$, then $(L; S) = (O; 1)$ and $(L; S_{(p,1)}) = (O; 2)$. 
It is thus easily checked that $d(L; S_{(p, 1)}) = 1$. 
See the definition of the curve graph for a $4$-holed sphere. 
Suppose that $n \geq 2$.   
Let $T^+_1, \ldots , T^+_{n-1}$ 
(respectively, $T^-_1, \ldots , T^-_{n-1}$) be the components of 
$B^+ \cap L$ (respectively, $B^- \cap L$) disjoint from $p$. 
We note that the arcs $T^+_1, \ldots , T^+_{n-1}$ 
(respectively, $T^-_1, \ldots , T^-_{n-1}$) remain to be components of 
$B^+_{(p,1)} \cap L$ (respectively, $B^-_{(p,1)} \cap L$). 
Let $T^-_{n+1}$ be the (unique) component of 
$( B^-_{(p,1)} \cap L ) - \bigcup_{i=1}^{n-1} T^-_{i}$ disjoint from $p$. 
Then there exist disjoint disks $Z^+_1 \subset B^+_{(p,1)}$ and 
$Z^-_{n+1} \subset B^-_{(p,1)}$ such that 
$Z^+_1 \cap L = \partial Z^+_1 \cap L = T^+_1$, 
$\partial Z^+_1 - T^+_1 \subset S_{(p,1)}$, 
$Z^-_{n+1} \cap L = \partial Z^-_{n+1} \cap L = T^-_{n+1}$
 and 
$\partial Z^-_{n+1} - T^-_{n+1} \subset S_{(p,1)}$. 
The simple closed curve 
$\alpha := \partial N (Z^+_1 \cap S_{(p,1)}; S_{(p,1)})$ bounds 
a disk in $B^+_{(p,1)} - (B^+_{(p,1)} \cap L)$ while 
$\beta := \partial N (Z^-_{n+1} \cap S_{(p,1)}; S_{(p,1)})$ bounds 
a disk in $B^-_{(p,1)} - (B^-_{(p,1)} \cap L)$. 
Since both $\alpha$ and $\beta$ are disjoint essential simple closed curves 
in $\Cl (S_{(p,1)} - N (S_{(p,1)} \cap L; S_{(p,1)}))$, the distance $d (L; S_{(p,1)})$ is at most $1$.  
\end{proof}

Recall that a subset $Y$ of a geodesic metric space $X$ is said to be 
{\it $K$-quasiconvex} in X for a positive number $K>0$ if, for any
two points $x, y$ in $Y$, any geodesic segment in $X$ connecting $x$ and $y$ lies in the 
$K$-neighborhood of $Y$. 
We are going to discuss the quasiconvexity of $\mathcal{D}^+$ and $\mathcal{D}^-$ in $\mathcal{C} (S_L)$. 
In fact, the following holds. 
\begin{thm}
\label{thm:convexity}
There exist a constant $K > 0$ satisfying the following property. 
For any $n$-bridge decomposition of a link $L$ in $S^3$ with $n \geq 2$, 
the set $\mathcal{D}^{+}$ $($respectively, $\mathcal{D}^{-})$ is $K$-quasiconvex in $\mathcal{C}(S_{L})$. 
\end{thm}

This theorem can actually be proved using 
Masur-Minsky \cite{MasurMinsky04} and 
Hamenst\"{a}dt \cite[Section 3]{Hamenstadt18}. 
In the following we adopt Vokes's arguments in \cite{Vokes18} and \cite{Vokes19} 
to explain that the $K$ in the above theorem can be taken to be at most $1796$.

In what follows, we assume that 
simple closed curves in a surface are properly embedded, essential, 
and their intersection is transverse and minimal up to isotopy. 

A compact orientable genus-$g$  surface $\Sigma$ with $m$ holes is said to be 
{\it non-sporadic} if $3g + m \geq 5$.  
Let $\alpha$ and $\beta$ be simple closed curves in a non-sporadic surface $\Sigma$. 
A simple closed curve $\gamma$ in $\Sigma$ is called an 
{\it $(\alpha, \beta)$-curve with $0$ corners} if $\gamma = \alpha$ or $\gamma = \beta$. 
A simple closed curve $\gamma$ in $\Sigma$ is called an 
{\it $(\alpha, \beta)$-curve with $2$ corners} if 
there exist subarcs $\alpha' \subset \alpha$ and $\beta' \subset \beta$ such that 
$\partial \alpha' = \partial \beta'$, 
$\Int \alpha' \cap \Int \beta' = \emptyset$, and  
$\gamma$ is homotopic in $\Sigma$ to the concatenation $\alpha' \ast \beta'$, 
where orientations are chosen in an appropriate way.  
A simple closed curve $\gamma$ in $\Sigma$ is called an 
{\it $(\alpha, \beta)$-curve with $4$ corners} if 
there exist subarcs $\alpha'_1, \alpha'_2  \subset \alpha$ and $\beta'_1, \beta'_2 \subset \beta$ satisfying the following. 
\begin{itemize}
\item
$\Int \alpha'_1$, $\Int \alpha'_2$, $\Int \beta_1'$, $\Int \beta_2'$ are mutually disjoint, 
\item
$\partial \alpha'_1 \cup \partial \alpha'_2 = \partial \beta'_1 \cup \partial \beta'_2$ (we allow the case where 
$\partial \alpha'_1$ and $\partial \alpha'_2$ (and hence $\partial \beta'_1$ and $\partial \beta'_2$) share a single point),  
\item
$\alpha'_1 \cup \alpha'_2 \cup \beta'_1 \cup \beta'_2$ is connected,  
\item
$\gamma$ is homotopic in $\Sigma$ to the concatenation $\alpha'_1 \ast \beta'_1 \ast \alpha'_2 \ast \beta'_2$, 
where orientations are chosen in an appropriate way.  
\end{itemize}
We note that an $(\alpha, \beta)$-curve with at most $2$ corners are called a {\it bicorn curve} in 
Przytycki-Sisto \cite{PrzytyckiSisto17}. 
For any simple closed curves $\alpha$, $\beta$ in 
a non-sporadic surface $\Sigma$, 
let $\mathcal{L}_0 (\alpha, \beta)$ denote the full subgraph of $\mathcal{C}(\Sigma)$ spanned by 
the set of $(\alpha, \beta)$-curves with at most $4$ corners. 
For a positive number $h$, we set $R (h)  := m - 4h$, where $m > 0$ is the number defined by 
 $ 2h (6 + \log_2 (m+2) ) = m  $. 
 
The following proposition is a part of Proposition 3.1 of Bowditch \cite{Bowditch14} 
that is necessary for our arguments. 
\begin{prop}
\label{prop:Bowditch's quasiconvexity condition}
Let $h > 0$ be a constant. 
Let $G$ be a connected graph equipped with the simplicial metric $d_G$. 
Suppose  that each pair $\{x, y\}$  $($possibly $x=y$$)$ of vertices of $G$ is associated with 
a connected subgraph $\mathcal{L} (x, y) \subset G$ with $x, y \in \mathcal{L} (x, y)$ satisfying the following.  
\begin{enumerate}
\item
For any vertices $x, y, z$ of $G$, it holds $\mathcal{L} (x, y) \subset N_h (\mathcal{L} (x, z) \cup \mathcal{L} (y, z) )$. 
\item
For any vertices $x, y$ of $G$ with $d_G (x, y) \leq 1$, the diameter of $\mathcal{L} (x, y)$ is at most $h$. 
\end{enumerate}
Then for any vertices $x, y$ of $G$, the Hausdorff distance in $G$ between $\mathcal{L} (x, y)$ and 
any geodesic segment $c$ in $G$ connecting $x$ and $y$ is 
at most $R (h)$. 
\end{prop}

The following lemma is proved in 
\cite[Lemma 5.1.4]{Vokes18}. 

\begin{lem}
\label{lem:constant hprime}
There exists a constant $h_1 > 0$ such that for any simple closed curves $\alpha$, $\beta$ in 
any non-sporadic surface $\Sigma$, 
$N_{h_1} (\mathcal{L}_0 (\alpha, \beta))$ is connected.   
 \end{lem}
We note that by the remark just before Lemma 5.1.12 in \cite{Vokes18} 
that the constant $h_1$ in the above lemma can be taken to be at most $7$. 

The next lemma is due to \cite[Lemma 5.1.5]{Vokes18}. 
\begin{lem}
\label{lem:constant h}
There exists a constant $h_2 > 0$ such that for any simple closed curves $\alpha$, $\beta$, $\gamma$ in 
any non-sporadic surface $\Sigma$, 
it holds $\mathcal{L}_0 (\alpha, \beta) \subset 
N_{h_2} (\mathcal{L}_0 (\alpha, \gamma) \cup \mathcal{L}_0 (\beta, \gamma))$.   
 \end{lem}
By \cite[Lemma 5.1.12]{Vokes18}, 
the constant $h_2$ in the above lemma can be taken to be at most $18$. 
For any simple closed curves $\alpha$, $\beta$ in 
a non-sporadic surface $\Sigma$, 
set $\mathcal{L} (\alpha, \beta) := N_{h_1} (\mathcal{L}_0 (\alpha, \beta))$. 
By \cite[Claim 10.4.2]{Vokes19}, for the constant $h_0 := 2 h_1 + h_2$, which can be taken to be at most 
$2 \cdot 7 +18 = 32$, the curve graph $\mathcal{C} (\Sigma)$ endowed with the associated subgraphs  
$\mathcal{L} (\alpha, \beta) $ for $\alpha, \beta \in \mathcal{C}^{(0)} (\Sigma)$ satisfies the 
condition of Proposition \ref{prop:Bowditch's quasiconvexity condition}. 
Using this fact, \cite[Lemma 10.4.4]{Vokes19} shows the following. 
\footnote{Precisely speaking, \cite[Lemma 10.4.4]{Vokes19} considers only the case of 
bicorn curves  in a closed orientable surface instead of 
curves with at most four corners in a non-sporadic orientable surface (possibly with boundary). 
Lemma \ref{lem:constant 2R plus 2}, however, can be proved in exactly the same way.} 

\begin{lem}
\label{lem:constant 2R plus 2}
Let $\alpha$, $\beta$ be simple closed curves in a non-sporadic surface $\Sigma$, 
$P$ a path in $\mathcal{L} (\alpha, \beta)$ from $\alpha$ and $\beta$. 
Then any geodesic segment in $\mathcal{C} (\Sigma)$ connecting $\alpha$ and $\beta$ lies in 
the $(2R(h_0) + 2)$-neighborhood of $P$.  
\end{lem}
We note that the minimum integer greater than or equal to $R ( h_0 )$ is $897$. 
Therefore, the above constant $2R(h_0) + 2$ can be taken to be at most $1796$.

Now, we quickly review the notion of the disk surgery. 
Let $D$ and $E$ be properly embedded disks in a $3$-manifold $M$ intersecting transversely and minimally. 
Suppose that $D \cap E \neq \emptyset$. 
 Let $E'$ be an outermost subdisk of $E$ cut off by $D \cap E$. 
 The arc $\partial E' \cap D$ cuts $D$ into two subdisks. 
 Choose one $D'$ of them. 
 Set $D_1 := D' \cup E'$. 
By a slight isotopy, the disk $D_1$ can be moved to be disjoint from $D \cup E'$.  
We call $D_1$ a {\it disk obtained by a surgery on D along $E$} (with respect to $E'$). 
An operation to obtain $D_1$ from $D$ in the above way is called a {\it disk surgery}. 
We note that the number $\#(D' \cap E)$ of components of $D' \cap E$ is less than $\#(D \cap E)$. 
Therefore, applying disk surgeries repeatedly, we obtain a finite sequence 
$D = D_0, D_1, D_2, \ldots, D_k = E$ of disks in $M$ such that 
$D_i \cap D_{i+1} = \emptyset$ for $i= 0, 1, \ldots , k-1$. 

\begin{proof}[Proof of Theorem~$\ref{thm:convexity}$]
Let $(L; S)$ be an $n$-bridge decomposition of a link $L$ in $S^3$ with $n \geq 2$. 
Let $\alpha$ and $\beta$ be arbitrary points in $\mathcal{D}^{+}$. 
Let $\{\gamma_{i}\}_{0 \le i \le s}$ be a sequence of vertices of $\mathcal{D}^{+}$   
obtained by disk surgery such that $\alpha=\gamma_{0}$ and $\beta=\gamma_{s}$.  
By definition the path $P$ in $\mathcal{C} (S_L)$ corresponding to this sequence 
lies within the set $\mathcal{L} (\alpha, \beta)$. 
It follows from Lemma \ref{lem:constant 2R plus 2} 
that any geodesic segment $c$ in $\mathcal{C}(S_{L})$ connecting $\alpha$ and $\beta$ 
lies within the $(2R(h_0) + 2)$-neighborhood of $P$. 
Since $\{\gamma_{i}\}_{0 \le i \le s} \subset \mathcal{D}^+$, the geodesic segment 
$c$ lies within the $(2R(h_0) + 2)$-neighborhood of $\mathcal{D}^+$, which implies the assertion. 
The argument for $\mathcal{D}^-$ is of course the same. 
\end{proof}

\subsection{Braid groups} 
\label{subsection_braidgroups}

Let $B_n$ be the (planar) braid group with $n$ strands. 
The group $B_n$ is generated by braids 
$\sigma_1, \sigma_2, \ldots, \sigma_{n-1}$ 
as shown in Figure \ref{fig_halftwist}(1). 
The product of braids is defined as follows. 
Given $b, b' \in B_n$, 
we stuck $b$ on $b'$, and concatenate the bottom  of $b$ 
with the top of $b'$. 
The product $bb' \in B_n$ is the resulting  braid, see Figure \ref{fig_halftwist}(2). 

\begin{center}
\begin{figure}[t]
\includegraphics[height=2.6cm]{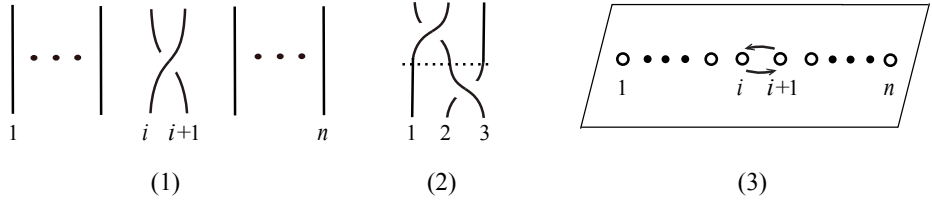}
\caption{(1) The generator $\sigma_i$ in the spherical/planar braid group. 
(2) The $3$-braid $\sigma_1 \sigma_2^{-1}$. 
(3) The half twist $h_i$ in the mapping class group on the sphere/disk with marked points.} 
\label{fig_halftwist}
\end{figure}
\end{center}

We set $\delta_j = \sigma_1 \sigma_2 \cdots \sigma_{j-1} \in B_n$. 
The {\it half twist} $\Delta  \in B_n$ is given by 
$$\Delta = \delta_n \delta_{n-1} \cdots \delta_2.$$ 
The second  power  $\Delta^2$ is called the {\it full twist}.

Let $ \SB_{n}$ be the spherical braid group with $n$ strands. 
By abusing notation, we still denote by $\sigma_i$, 
 the spherical braid as shown in Figure \ref{fig_halftwist}(1). 
We define the product of spherical braids in the same manner as above.

We recall connections between the planar/spherical braid groups and the mapping class groups 
on the disk/sphere with marked points. 
For more details, see \cite{Birman74}. 
Let $D_n$ be the disk with $n$ marked points. 
Then we have the surjective homomorphism 
$$\Gamma_D: B_n \rightarrow \MCG (D_n)$$
which sends each generator $\sigma_i$ to the right-handed half twist $h_i$ 
between the $i$th and $(i+1)$th marked points, see Figure \ref{fig_halftwist}(3). 
The kernel of $\Gamma_D$ is an infinite cyclic group generated by the full twist $\Delta^2 $, 
that is, 
$\mathrm{ker}\,\Gamma_D = \langle \Delta^2 \rangle$. 
Thus,  we have  
$$B_{n} / \langle \Delta^2 \rangle \cong \MCG (D_n).$$

We also have the surjective homomorphism 
$$\Gamma: \SB_n \rightarrow \MCG (\Sigma_{0,n}),$$ 
sending each generator $\sigma_i$ to the right-handed  half twist $h_i$ 
between the $i$th and $(i+1)$th marked points in the sphere. 
The kernel of $\Gamma$ is isomorphic to ${\Bbb Z}/{2{\Bbb Z}}$ 
which is generated by the full twist $\Delta^2 \in \SB_n$. 
Thus we have 
$$\SB_{n} / \langle \Delta^2 \rangle \cong \MCG (\Sigma_{0,n}).$$

\subsection{Wicket groups on trivial tangles}
\label{subsection_wicket-groups}

Let $\mathcal{T}$ be an $n$-tangle in $B^-$ and 
let $b \in \SB_{2n}$. 
We stuck $b$ on $\mathcal{T}$ concatenating the bottom endpoints of $b$ with the endpoints  $\partial {\mathcal{T}}$. 
Then we obtain an $n$-tangle ${}^{b}\mathcal{T}$ in $B^-$, 
see Figure \ref{fig_wicket2}(1). 
We may assume that 
${}^{b}\mathcal{T}$ share the common endpoints as ${}\mathcal{T}$. 
Observe that 
$\SB_ {2n}$ acts on $n$-tangles in $B^-$ from the left: 
$${}^{b'  b}\mathcal{T} = {}^{b'}({}^{b}\mathcal{T}) 
\hspace{2mm}
\mbox{for}\  b, b' \in \SB_{2n}.$$

\begin{center}
\begin{figure}[t]
\includegraphics[height=3.8cm]{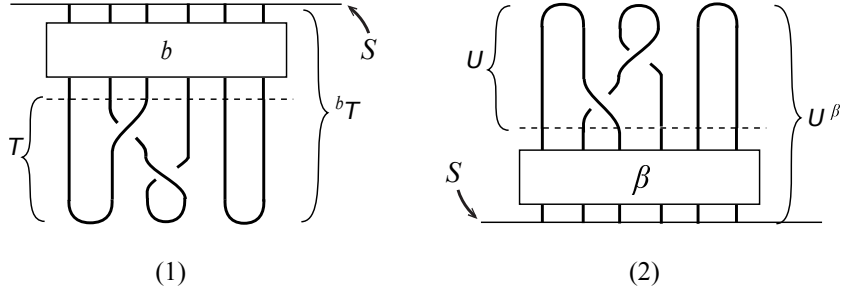}
\caption{
(1) 
The tangle ${}^{b}\mathcal{T}$ in the $3$-ball $B^{-}$, where 
$\mathcal{T} = {}^{\sigma_2 \sigma_3^{-1}} \mathcal{A}$. 
(2) 
The tangle $ \mathcal{U}^{\beta} $ in the $3$-ball $B^{+}$, where 
$\mathcal{U} = \Bar{\mathcal{A}}^{\sigma_3 \sigma_2^{-1}}$.} 
\label{fig_wicket2}
\end{figure}
\end{center}

Similarly, given an $n$-tangle $ \mathcal{U}$ in $B^+$ and a braid $\beta \in \SB_{2n}$, 
we obtain an $n$-tangle $\mathcal{U}^{\beta}$ in $B^+$, see Figure \ref{fig_wicket2}(2). 
Then $\SB_ {2n}$ acts on $n$-tangles in $B^+$ from the right: 
$$\mathcal{U}^{\beta'  \beta} = (\mathcal{U}^{\beta'})^{\beta} \hspace{2mm} \mbox{for} \ \beta, \beta' \in \SB_{2n}.$$

Recall the involution 
$\rho: S^3 \rightarrow S^3$ defined in Section 
\ref{subsection_Bridge-decompositions}. 
Let $\mathcal{A}$ and $\Bar{\mathcal{A}}$ be $n$-tangles in $B^-$ and $B^+$ respectively as before. 
By definition of $\rho$, we have $\rho (x, y, z) = (x, y, -z)$ for $(x, y, z) \in \R^3$. This implies that
\begin{eqnarray*}
\rho({}^b \mathcal{A})&=&\Bar{\mathcal{A}}^{b^{-1}}, 
\\
\rho( \Bar{\mathcal{A}}^{\beta}) &=& {}^{\beta^{-1}} \mathcal{A}. 
\end{eqnarray*}
For example, $\rho$ interchanges ${}^{\sigma_2 \sigma_3^{-1}} \mathcal{A}$ with 
$\Bar{\mathcal{A}}^{\sigma_3 \sigma_2^{-1}}$.

\begin{remark}
\label{rem_extension}
Given  $\phi \in \MCG (\Sigma_{0,2n})$, 
there is a braid $b_{\phi} \in \SB_{2n}$ such that 
$\Gamma(b_{\phi}) = \phi$. 
Take any representative $f: \Sigma_{0,2n} \rightarrow \Sigma_{0,2n}$ of $\phi$. 
We regard $f$ as an orientation-preserving homeomorphism on the sphere 
$S= \partial B^- = \partial B^+$ 
with $2n$ marked points. 
Let $\Phi: B^- \rightarrow B^-$ be an extension of $f$, that is, 
$\Phi |_{\partial B^-} = f$. 
Let $ \mathcal{A}$ be the standard $n$-tangle in $B^{-}$. 
Assume that $\partial \mathcal{A} \subset S$ equals the set of $2n$ marked points. 
Then $\Phi (\mathcal{A}) $ is an $n$-tangle in $B^-$.  
It follows that 
$$\Phi (\mathcal{A}) = {}^{b_{\phi}} \mathcal{A}. $$
for we equip $S$ with the orientation induced by that of $B^-$. 
Let us turn to the homeomorphism $\rho \Phi \rho |_{B^+}: B^+ \rightarrow B^+$. 
Since $\rho|_S=  \id_S$, 
it holds $\rho \Phi \rho |_{\partial B^+} = f$. 
Hence  $\rho \Phi \rho |_{B^+} : B^+ \rightarrow B^+$  
is an extension of $f$ over $B^+$. 
Then we have 
$$\rho \Phi \rho(\Bar{\mathcal{A}}) = \rho \Phi (\mathcal{A}) 
= \rho ({}^{b_{\phi}} \mathcal{A}) = \Bar{\mathcal{A}}^{b_{\phi}^{-1}}.$$
\end{remark}

From the above discussion, it is easy to see the following lemma.

\begin{lem} 
\label{lem_trivial-tangle}
Let $\mathcal{T}$ be an $n$-tangle in $B^-$, and let 
$\mathcal{U}$ be an $n$-tangle in $B^+$. 
\begin{enumerate}
\item[(1)] 
$\mathcal{T}$  is trivial if and only if 
$\mathcal{T} = {}^b \mathcal{A}$ for some $b \in \SB_{2n}$.

\item[(2)] 
$ \mathcal{U}$ is trivial if and only if 
$\mathcal{U} = \Bar{\mathcal{A}}^{\beta}$ for some $\beta \in \SB_{2n}$.  
\end{enumerate}
\end{lem}

%

For a trivial $n$-tangle $ \mathcal{T}$ in $B^-$, 
we define a subgroup $\SW_{2n}(\mathcal{T}) \subset \SB_{2n}$ as follows. 
$$\SW_{2n} (\mathcal{T}) =
\{ b \in \SB_{2n} \, | \, ^{b}\mathcal{T}=\mathcal{T}\}.$$
Since $\Delta^2 \in \SW_{2n}(\mathcal{T})$, we have 
$$\mathrm{ker}\,\Gamma = \langle \Delta^2 \rangle \subset \SW_{2n}(\mathcal{T}).$$
The group  $\SW_{2n}(\mathcal{A})$ for the standard $n$-tangle $\mathcal{A}$  is called the {\it wicket group}. 
We write $$\SW_{2n}= \SW_{2n}(\mathcal{A}).$$
See Brendle-Hatcher \cite{BrendleHatcher13} for more study on the wicket groups. 

Since the map 
$$\MCG  (B^-, \mathcal{A}) \rightarrow 
\MCG  (\partial B^-, \partial \mathcal{A})$$ 
sending $[f] \in \MCG  (B^-, \mathcal{A})$ to 
$[ f |_{\partial B^-} ] \in \MCG  (\partial B^-, \partial \mathcal{A})$ 
is injective, we regard  $\MCG  (B^- , \mathcal{A}) $ as a subgroup of $\MCG  (\partial B^-, \partial \mathcal{A})$ ($= \MCG  (\Sigma_{0, 2n})$). 
The following theorem is proved in \cite[Theorem~2.6]{HiroseKin17}. 

\begin{thm}
\label{thm:Kin-Hirose Theorem 2.6}
It holds $ \Gamma  (\SW_{2n} ) = \MCG  (B^-, \mathcal{A})$. 
Thus, we have 
$\SW_{2n} / \langle \Delta^2 \rangle \cong \MCG  (B^-, \mathcal{A})$. 
\end{thm}

\begin{ex}
\label{ex_SWA}
We define  $x, y, z \in \SB_ 6$ as follows. 
(See  Figure \ref{fig_ex-wicket}.)
\begin{eqnarray*}
x&=&   \sigma_{3}^2 \sigma_2 \sigma_{3}^2 \sigma_2, 
\\
y&=&  \sigma_1^2 \sigma_2 \sigma_3 \sigma_4 \sigma_5 \sigma_1 \sigma_2 \sigma_3 \sigma_4,  
\\
z&=& \sigma_1^2 \sigma_2 \sigma_3 \sigma_4  \sigma_1 \sigma_2 \sigma_{3} \sigma_{3} \sigma_{4}.
\end{eqnarray*}
The tangles 
${}^{x} \mathcal{A}$, $ {}^{y} \mathcal{A}$ and ${}^{z} \mathcal{A}$ are equivalent to $ \mathcal{A}$. 
Hence  $x,y,z \in \SW_6$. 
\end{ex}

\begin{center}
\begin{figure}[t]
\includegraphics[height=3cm]{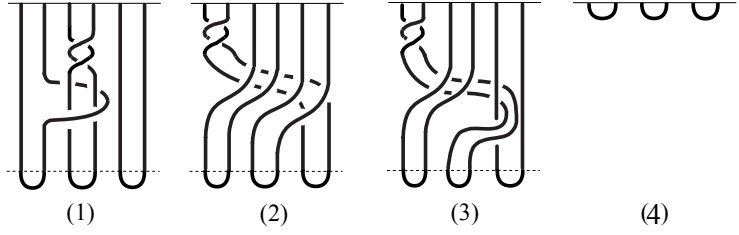}
\caption{(1) ${}^{x} \mathcal{A}$; (2) ${}^{y} \mathcal{A}$; and (3) ${}^{z} \mathcal{A}$ 
for $x,y,z \in \SB_{6}$ in Example \ref{ex_SWA}. 
They are equivalent to  (4) $\mathcal{A}$.}
\label{fig_ex-wicket}
\end{figure}
\end{center}

\begin{lem}
\label{lem_wicket-con} 
Let $\mathcal{T}$ be a trivial $n$-tangle in $B^-$. 
Let $b$ be a spherical $2n$-braid such that 
$\mathcal{T}= {}^{b}\mathcal{A}$. 
Then we have 
$\SW_{2n}(\mathcal{T})= b(\SW_{2n})b^{-1}$. 
\end{lem}

\begin{proof}
Take an element $\beta \in \SW_{2n}(\mathcal{T})= \SW_{2n} ({}^{b}\mathcal{A})$. 
By definition of $\SW_{2n} ({}^{b}\mathcal{A})$, we have 
${}^{\beta}({}^{b}\mathcal{A}) = {}^{b}\mathcal{A}$. 
Then 
${}^{b^{-1} \beta b}\mathcal{A} = \mathcal{A}$. 
Hence we have 
$b^{-1} \beta b \in \SW_{2n}$, which says that 
$\beta \in b (\SW_{2n}) b^{-1}$. 
Thus $\SW_{2n}(\mathcal{T}) \subset b(\SW_{2n})b^{-1}$. 
The proof of $b(\SW_{2n})b^{-1} \subset \SW_{2n}(\mathcal{T})$ is similar. 
\end{proof}

For trivial $n$-tangles $ \mathcal{T}$ and $ \mathcal{U}$ in $B^-$, 
we set 
$$\SW_{2n} (\mathcal{T}, \mathcal{U}) = \SW_{2n} (\mathcal{T}) \cap \SW_{2n} (\mathcal{U}).$$ 
We call the group $\SW_{2n} (\mathcal{T}, \mathcal{U})$ the {\it wicket group on} $\mathcal{T}$ {\it and} $\mathcal{U}$. 
Obviously, 
$\SW_{2n}(\mathcal{A}, \mathcal{A}) =\SW_{2n}(\mathcal{A}) (= \SW_ {2n}).$

The following lemma will be used in the proofs of Theorems~\ref{introthm_asymptitic behavior for the trivial knot} and \ref{introthm_asymptotic behavior for the Hopf link}.

\begin{lem}
\label{lem_SWBC}
Let $x,y$ and $z$ be elements of $\SW_6$ 
as in  Example $\ref{ex_SWA}$. 
Let $\mathcal{B}= \mathcal{B}_3$ and $\mathcal{C}= \mathcal{C}_3$ 
be the trivial tangles as in Figure $\ref{fig_BandC}$. 
Then we have 
$x, y \in \SW_6(\mathcal{B})$ and 
$x, z \in \SW_6(\mathcal{C})$. 
In particular $x, y \in \SW_6(\mathcal{A}, \mathcal{B})$ and 
$x, z \in \SW_6(\mathcal{A}, \mathcal{C})$.  
\end{lem}

\begin{proof}
We see that 
${}^{x} \mathcal{B}= \mathcal{B}=  {}^{y} \mathcal{B}$, and 
${}^{x} \mathcal{C}= \mathcal{C}= {}^{z} \mathcal{C}$ (Figure \ref{fig_ex-wicketBC}).  
We are done. 
\end{proof}

\begin{center}
\begin{figure}[t]
\includegraphics[height=4.6cm]{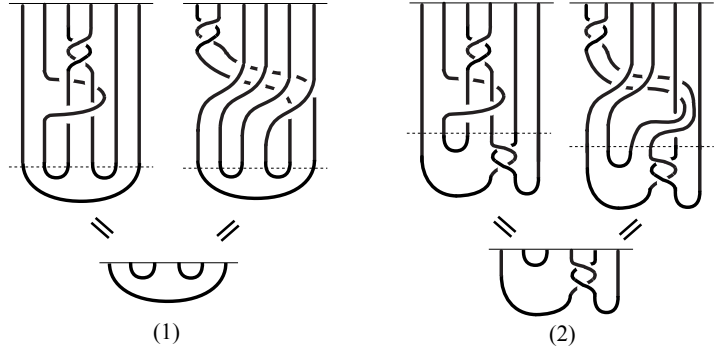}
\caption{(1) ${}^{x} \mathcal{B}= \mathcal{B}=  {}^{y} \mathcal{B}$ and 
(2) ${}^{x} \mathcal{C}= \mathcal{C}= {}^{z} \mathcal{C}$ 
for $x,y,z \in \SB_{6}$ in Example \ref{ex_SWA}. 
See also Figure \ref{fig_BandC}.}
\label{fig_ex-wicketBC}
\end{figure}
\end{center}

By Lemma~\ref{lem_wicket-con}, we immediately have the following corollary.

\begin{cor}
\label{cor_double_wicket_char}
Let $b,d \in \SB_{2n}$. 
Then we have 
$$
\SW_{2n} ({}^{b}\mathcal{A}, {}^{d}\mathcal{A})
=
b (\SW_{2n} ) b^{-1} \cap d (\SW_{2n} ) d^{-1}. 
$$
\end{cor}

\begin{lem}
\label{lem_wicket-conj}
Let $ \mathcal{T}$ and $ \mathcal{U}$ be trivial $n$-tangles in $B^-$. 
Let $b$ and $d$ be the spherical $2n$-braids such that 
$\mathcal{T} = {}^{b}\mathcal{A}$ and 
$\mathcal{U} = {}^{d}\mathcal{A}$. 
Then we have 
$$\SW_{2n} (\mathcal{T}, \mathcal{U})  = b( \SW_{2n} (\mathcal{A},  {}^{b^{-1}d} \mathcal{A})) b^{-1}.$$ 
\end{lem}

\begin{proof}
By Corollary  \ref{cor_double_wicket_char}, we have 
$$\SW_{2n} (\mathcal{T}, \mathcal{U}) = 
b (\SW_{2n}) b^{-1} \cap d (\SW_{2n}) d^{-1} .$$
By Corollary \ref{cor_double_wicket_char} again, we have 
$$\SW_{2n}(\mathcal{A}, {}^{b^{-1} d} \mathcal{A})
= \SW_{2n} \cap b^{-1} d (\SW_{2n}) d^{-1} b.$$
These two equalities imply the assertion. 
\end{proof}

\subsection{Hyperelliptic handlebody groups}

In this subsection, we review the notion of the hyperelliptic handlebody group 
developed in \cite{HiroseKin17}. 
We go into some details in some of its properties for the later use.

Let $V_g$ be a handlebody of genus $g \geq 2$ with $\partial V_g = \Sigma_g$. 
We call an involution $\hat{\iota} \in \Homeo_+(V_g)$ a {\it hyperelliptic involution} of $V_g$ 
if so is $\hat{\iota} |_{\Sigma_g}$ of $\Sigma_g$,  that is, $\hat{\iota}|_{\Sigma_g}$ is an order 
$2$ element of $\Homeo_+ (\Sigma_g)$ that acts on $H_1 (\Sigma_g; \Z)$ by $-I$. 
The following lemma is straightforward from Pantaleoni-Piergallini \cite{PantaleoniPiergallini_11}.

\begin{lem}
\label{lem:Pantaolini-Piergallini}
Any two  hyperelliptic involutions of $V_g$ are conjugate in the handlebody group $\MCG  (V_g)$. 
\end{lem}

By this lemma, without loss of generality we can assume that $\hat{\iota}$ is the map shown in 
Figure \ref{fig_iota}, where we think of $V_g$ as being embedded in $\R^3$. 

\begin{center}
\begin{figure}[t]
\includegraphics[height=2.5cm]{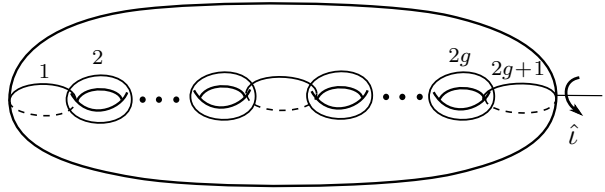}
\begin{picture}(400,0)(0,0)
\put(79,55){\begin{scriptsize}$1$\end{scriptsize}}
\put(99,60){\begin{scriptsize}$2$\end{scriptsize}}
\put(234,61){\begin{scriptsize}$2g$\end{scriptsize}}
\put(250,56){\begin{scriptsize}$2g \! + \! 1$\end{scriptsize}}
\put(279,30){$\hat{\iota}$}
\end{picture}
\caption{A hyperelliptic involution $\hat{ \iota}$ defined on $V_g$.}
\label{fig_iota}
\end{figure}
\end{center}

Fix a hyperelliptic involution $\hat{\iota}$ of $V_g$ and set $\iota := \hat{\iota} |_{\Sigma_g}$. 
We denote by $\SHomeo_+ (V_g)$ (respectively, $\SHomeo_+ (\Sigma_g)$) the centralizer 
in $\Homeo_+ (V_g)$ (respectively, $\Homeo_+ (\Sigma_g)$) of $\hat{\iota}$ (respectively, $\iota$). 
We call  
$$\mathcal{H} (V_g) := \pi_0 (\SHomeo_+ (V_g))$$ a {\it hyperelliptic handlebody group}, 
and 
$$\mathcal{H} (\Sigma_g) := \pi_0 (\SHomeo_+ (\Sigma_g))$$ a {\it hyperelliptic mapping class group}. 
By Birman-Hilden \cite{BirmanHilden71} the following holds.
\begin{thm}
\label{thm:Birman-Hilden}
We have a canonical isomorphism 
$$\mathcal{H} (\Sigma_g)/ \langle [ \iota ] \rangle 
\cong \MCG (\Sigma_{0,2g+2}).  $$
\end{thm}
We note that this theorem, together with \cite[Theorem 4]{BirmanHilden73}, implies that  
the group $\mathcal{H} (\Sigma_g)$ can naturally be identified with the centralizer in 
$\MCG (\Sigma_g)$ of the mapping class $[\iota]$. 
By Theorem~\ref{thm:Birman-Hilden} the quotient map 
$\mathcal{H}(\Sigma_{g}) \to \mathcal{H}(\Sigma_{g})/ \langle [ \iota ] \rangle$ 
gives the surjective homomorphism 
$$\Pi : \mathcal{H}(\Sigma_{g}) \to \MCG (\Sigma_{0,2g+2}).$$ 
The map $\Pi$ sends  $\tau_i$ to the half twist $h_i $, 
where $\tau_i$ is the  right-handed Dehn twist about the simple closed curve labeled with the number $i$ in Figure \ref{fig_iota}. 

Let $q: V_g \to V_g / \hat{\iota} =: B$ be the projection. 
Set $\mathcal{A} ( = \mathcal{A}_{g+1}) := q (\Fix (\hat{\iota}))$, 
where $\mathrm{Fix} ( ~ \cdot ~ )$ denotes the set of fixed points. 
We note that $\mathcal{A}$ is the trivial $(g + 1)$-tangle in the $3$-ball $B$. 
By basic algebraic topology arguments, we have the following. 

\begin{lem}
\label{lem:lifting a homeomorphism}
\begin{enumerate}
\item
Any element in $\Homeo_+ (B, \mathcal{A})$ lifts to an element of $\SHomeo_+ (V_g)$. 
\item 
Given a path in $\Homeo_+ (B, \mathcal{A})$ with the initial point 
$\hat{\phi} \in \Homeo_+ (B, \mathcal{A})$ 
and a lift $\hat{f} \in \SHomeo_+ (V_g)$ of $\hat{\phi}$, 
there exists a unique lift in $\SHomeo_+ (V_g)$ of the path with the initial point 
$\hat{f}$.
\end{enumerate}
\end{lem}

The next theorem follows from Theorem~\ref{thm:Birman-Hilden} and 
Lemma~\ref{lem:lifting a homeomorphism}.

\begin{thm}[Theorem~2.11 in \cite{HiroseKin17}]
\label{thm:H(V)/i}
The natural map 
$$ \mathcal{H}(V_{ g}) \to \MCG  (B, \mathcal{A}_{g + 1}) $$
is surjective and its kernel is $\langle [\hat{\iota}] \rangle$. 
Thus, we have 
$\mathcal{H}(V_{g}) / \langle [\hat{\iota}] \rangle \cong \MCG  (B, \mathcal{A}_{g+1})$. 
\end{thm}

We denote by  
$\mathrm{E} \Homeo_+ (\Sigma_g)$  the subgroup of $\Homeo_+ (\Sigma_g)$ consisting of 
homeomorphisms of  $\Sigma_g$ that extend to those of $V_g$. 
Note that the injectivity of the natural map $\MCG (V_g) \to \MCG (\Sigma_g)$ 
implies 
$$\pi_0 (\mathrm{E} \Homeo_+ (\Sigma_g)) \cong \MCG (V_g).$$ 
We say that two elements of $\SHomeo_+ (V_g)$ (respectively, $\SHomeo_+ (\Sigma_g)$) 
are {\it symmetrically isotopic} 
if they lie in the same component of $\SHomeo_+ (V_g)$ (respectively, $\SHomeo_+ (\Sigma_g)$).  

\begin{lem}
\label{lem:Hirose-Kin 2017 Prop A.6}
\begin{enumerate}
\item
For any $f \in \mathrm{E} \Homeo_+ (\Sigma_g) \cap \SHomeo_+ (\Sigma_g)$, 
there exists an element $\hat{f} \in \SHomeo_+ (V_{ g })$ with $\hat{f} |_{\Sigma_g} = f$. 
\item
Let $\hat{f}_0$ and $\hat{f}_1$ be elements of $\SHomeo_+ (V_g)$ 
that are isotopic. 
Then $\hat{f}_0$ and $\hat{f}_1$ are symmetrically isotopic.  
\end{enumerate}
\end{lem}
\begin{center}
\begin{figure}[t]
\includegraphics[width=10cm]{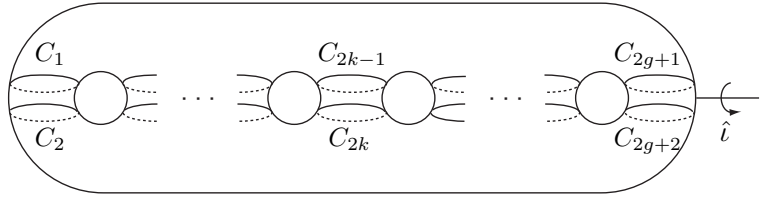}
\begin{picture}(400,0)(0,0)
\put(48,63){\begin{small}$C_1$\end{small}}
\put(48,31){\begin{small}$C_2$\end{small}}
\put(155,63){\begin{small}$C_{2k-1}$\end{small}}
\put(159,31){\begin{small}$C_{2k}$\end{small}}
\put(267,63){\begin{small}$C_{2g+1}$\end{small}}
\put(267,31){\begin{small}$C_{2g+2}$\end{small}}
\put(307,33){$\hat{\iota}$}
\end{picture}
\caption{The simple closed curves $C_1, \ldots, C_{2g+2}$ on $\Sigma_g$.} 
\label{fig_symmetric_loops}
\end{figure}
\end{center}
\begin{proof}
(1) 
Choose an arbitrary map $f \in \mathrm{E} \Homeo_+ (\Sigma_g) \cap \SHomeo_+ (\Sigma_g)$. 
Let $C_1, \ldots , C_{2g+2}$ be the simple closed curves on $\Sigma_g$ 
with $\iota (C_{2k-1}) = C_{2k}$ $(k = 1, \ldots , g+1)$ as shown in 
Figure \ref{fig_symmetric_loops}. 
There exist disks $D_1 , \ldots , D_{2g+2}$ in $V_g$ such that $\partial D_i = C_i$ 
$(i=1, \ldots, 2g+2)$, 
$D_i \cap D_j = \emptyset $ $(i \neq j)$, and $\hat{\iota} (D_{2k-1}) = D_{2k}$ 
$(k=1, \ldots , g+1)$.  
Let $A_k \subset \Sigma_g$ be the parallelism region between $C_{k-1}$ and $C_k$. 
Let $B_k$ be the 3-ball bounded by 
$D_{2k-1} \cup A_k \cup D_{2k}$. 
Set $C'_i := f (C_i)$. 
We note that it holds 
$$
\iota (C'_{2k-1}) = \iota (f (C_{2k-1}) ) = f (\iota (C_{2k-1})) = f(C_{2k}) 
= C'_{2k}
$$
for each $k=1, \ldots , g+1$, and 
$$
f (\mathrm{Fix} (\iota)) = f (\iota (\mathrm{Fix} ( \iota ))) 
= \iota (f (\mathrm{Fix} ( \iota ) )). 
$$
The latter implies that $f (\mathrm{Fix} (\iota)) = \mathrm{Fix} (\iota)$. 
By Edmonds \cite{Edmonds86} there exist disks 
 $D_{1}', D_{3}' , \ldots , D'_{2g+1}$ in $V_g$ such that 
$\partial D'_{2k-1} = C'_{2k-1}$, $D'_{2k-1} \cap \mathrm{Fix}(\hat{\iota}) = \emptyset $, and  
$D'_{2k-1} \cap \hat{\iota}(D'_{2k-1}) = \emptyset$ $(k = 1, \ldots , g+1)$. 
Set $D'_{2k} = \hat{\iota} (D'_{2k-1})$. 
We note that $\partial D'_{2k} = C'_{2k}$. 
Since $f$ preserves $\mathrm{Fix} (\iota)$, 
the self-homeomorphism $\phi$ of $\partial B$ induced from $f$ preserves 
$\partial \mathcal{A}$, that is, $\phi$ is an element of 
$\Homeo_+ (\partial B, \partial \mathcal{A})$. 
Set $E_k :=  q (D_{2k-1}) =  q (D_{2k})$ and 
$E'_k := q (D'_{2k-1}) = q (D'_{2k})$ 
$(k=1, \ldots , g+1)$. 
Since $\partial E'_i \cap \partial E'_j = \emptyset$ ($i \neq j$), 
the intersection of $E'_i$ and $E'_j$ consists only of simple closed curves. 
Since $B \setminus \mathcal{A}$ is irreducible, we can move 
$E'_1 , \ldots , E'_{g+1}$ by an isotopy in $B \setminus \mathcal{A}$ so as to satisfy 
$E'_i \cap E'_j = \emptyset$ ($i \neq j$). 
The disk $E_k$ cuts off from $B$ the 3-ball $q(B_k)$ that contains the single 
component $ q(B_k) \cap \mathcal{A}$ of $\mathcal{A}$. 
Similarly, the disk $E'_k$ cuts off from $B$ the 3-ball $q( f (B_k ))$ that contains 
the single component 
$ q( f (B_k)) \cap \mathcal{A} $ of $\mathcal{A}$. 
Therefore, we can find an extension 
$\hat{\phi} \in \Homeo_{+} (B, \mathcal{A})$ of 
$\phi$ with 
$\hat{\phi} (E_k) = E'_k$.  
By Lemma~\ref{lem:lifting a homeomorphism}, $\hat{\phi}$ 
lifts to $\hat {f} \in \SHomeo+ (V_g)$. 
By replacing $\hat{f}$ with $\hat{\iota} \circ \hat{f}$, if necessary, 
we have $\hat{f}|_{\Sigma_g} = f$.  

\noindent (2)
For $i = 0,1$, let $\hat{\phi}_i$ be the elements in $\Homeo_+ (B, \mathcal{A})$ induced from 
$\hat{f}_i$.  
Set 
$\phi_i := \hat{\phi}_i |_{\partial B} \in \Homeo_+ (\partial B, \partial \mathcal{A})$. 
By Theorem~\ref{thm:Birman-Hilden} there exists a symmetric isotopy from 
 $\hat{f}_0 |_{\Sigma_g} $ to  
$\hat{f}_1 |_{\Sigma_g}$.  
This isotopy induces a path in $\Homeo_+ (\partial B, \partial \mathcal{A})$ from $\phi_0$
to $\phi_1$. 
By Proposition~A.4 in \cite{HiroseKin17}, there exists a path in $\Homeo_+ (B,  \mathcal{A})$ 
from $\hat{\phi}_0$ to $\hat{\phi}_1$. 
By Lemma~\ref{lem:lifting a homeomorphism}(2), this path lifts to a path in $\SHomeo_+(V_g)$ 
with the initial point $\hat{f}_0$. 
Since $\hat{\iota}$ is an involution, the terminal point of the path is either $\hat{f}_1$ or $\hat{\iota} \circ \hat{f}_1$. 
The latter is impossible as we assumed that $\hat{f}_0$ and $\hat{f}_1$ are isotopic, and so 
 $\hat{f}_0$ and $\hat{\iota} \circ \hat{f}_1$ cannot be isotopic. 
This completes the proof. 
\end{proof}



 
Recall that both $\mathcal{H} (\Sigma_g) (= \pi_0 ( \SHomeo_+ \Sigma_g )) $ and 
$\MCG  (V_g) ( \cong \pi_0 ( \mathrm{E} \Homeo_+ (\Sigma_g)))$ can 
be regarded as subgroups of $\MCG (\Sigma_g)$.

\begin{thm}
\label{thm:hyperelliptic handlebody group as an intersection}
The map 
$$ \SHomeo_+ (V_g) \to \mathrm{E} \Homeo_+ (\Sigma_g) \cap \SHomeo_+ (\Sigma_g) $$ 
that takes $\hat{f} \in \SHomeo_+ (V)$ to $\hat{f}|_{\Sigma_g}$ 
induces an isomorphism 
$$\mathcal{H} (V_g) \xrightarrow{\cong} \MCG  (V_g) \cap \mathcal{H} (\Sigma_g).$$ 
\end{thm}
\begin{proof}
Recall that $\pi_0 (\SHomeo_+ (V_g)) = \mathcal{H} (V_g)$.  
Note that any element of $\Homeo_+ (\Sigma_g)$ isotopic to an element of 
$\mathrm{E} \Homeo_+ (\Sigma_g)$ is contained in 
$\mathrm{E} \Homeo_+ (\Sigma_g)$. 
This fact, together with Theorem~\ref{thm:Birman-Hilden}, allows us 
to think of 
$ \pi_0 ( \mathrm{E} \Homeo_+ (\Sigma_g) \cap \SHomeo_+ (\Sigma_g) )$ 
as the subgroup of $\pi_0 (\mathrm{E} \Homeo_+ (\Sigma_g))$ 
consisting of the mapping classes that contain an element of 
$\SHomeo_+ (\Sigma_g)$. 
Therefore, we can identify 
$ \pi_0 ( \mathrm{E} \Homeo_+ (\Sigma_g) \cap \SHomeo_+ (\Sigma_g) )$ 
with $\MCG (V_g) \cap \mathcal{H} (\Sigma_g)$ in a natural way. 
%
The surjectivity and injectivity of 
the map $\mathcal{H} (V_g) \to \MCG  (V_g) \cap \mathcal{H} (\Sigma_g)$
now follow from 
Lemma~\ref{lem:Hirose-Kin 2017 Prop A.6}(1) and (2), respectively.  
\end{proof}

In consequence, 
we have the following canonical identifications from Theorems~\ref{thm:Kin-Hirose Theorem 2.6}, 
\ref{thm:H(V)/i} and \ref{thm:hyperelliptic handlebody group as an intersection}: 
$$ \mathcal{H} (V_g) = \MCG  (V_g) \cap \mathcal{H}(\partial V_g) = \Pi^{-1}(\Gamma(\SW_{2g+2})). $$
$$ \mathcal{H} (V_g) / \langle [ \hat{\iota} ] \rangle = \MCG  (B, \mathcal{A}_{g+1}) 
= \Gamma (\SW_{2g + 2}) = \SW_{2g + 2} / \langle  \Delta^2  \rangle  . $$

\section{Goeritz groups of bridge decompositions}
\label{section_Goeritz groups of bridge decompositions}

 Let $(L; S)= (B^+ \cap L) \cup_S (B^- \cap L)$ be an $n$-bridge decomposition 
of a link $L \subset S^3$. 
We define the {\it Goeritz group}, denoted by 
$\mathcal{G}(L; S)$ or $\mathcal{G}((B^+ \cap L) \cup_S (B^- \cap L))$,  
of the $n$-bridge decomposition by 
$$ \mathcal{G}(L; S) = \MCG  (S^3, B^+, L) . $$ 
Since the map 
$$\mathcal{G}(L; S) \rightarrow \MCG  (S, S \cap L ) \cong \MCG  
(\Sigma_{0, 2n})$$ 
sending $[f] \in \mathcal{G} (L; S) $ to $[f|_{S}] \in  \MCG  (S, S \cap L)  $ is injective, 
we regard $\mathcal{G}(L; S)$ as  a subgroup of $ \MCG (\Sigma_{0, 2n})$. 
When $(L; S)$ is {a} unique $n$-bridge decomposition of $L$ up to equivalence, 
we simply call $\mathcal{G}(L; S) $ the {\it $n$-bridge Goeritz group} of $L$, 
and we denote it by $\mathcal{G}(L; n)$.  
In this section, we discuss several basic properties of this group.

\subsection{Relation to wicket groups on tangles}

Let $(B^+ \cap L) \cup_S (B^- \cap L)$ be an $n$-bridge decomposition of a link $L$. 
Then there exist braids $b, d \in \SB_{2n}$ such that 
$$(B^+ \cap L) \cup_S (B^- \cap L) =  \Bar{\mathcal{A}}^{d} \cup_S  {}^{b}\mathcal{A}.$$
We now describe Goeritz groups of the bridge decompositions in terms of wicket groups on tangles.

\begin{thm}
\label{thm_bridge_char}
Let $b,d \in \SB_{2n}$. 
Then we have 
$$
\mathcal{G}(\Bar{\mathcal{A}}^{d} \cup_S  {}^{b}\mathcal{A}) = 
\Gamma \bigl( \SW_{2n}({}^{d^{-1}}\mathcal{A}, {}^{b}\mathcal{A})\bigr) 
\cong 
\SW_{2n}({}^{d^{-1}}\mathcal{A}, {}^{b}\mathcal{A}) / \langle \Delta^2 \rangle. $$
\end{thm}

\begin{proof}
We prove that 
$\mathcal{G}(\Bar{\mathcal{A}}^{d} \cup_S  {}^{b}\mathcal{A})  \subset 
\Gamma \bigl( \SW_{2n}({}^{d^{-1}}\mathcal{A}, {}^{b}\mathcal{A})\bigr)$. 
Take an element 
$\phi \in \mathcal{G} (\Bar{\mathcal{A}}^{d} \cup_S  {}^{b}\mathcal{A}) $. 
Then there exists a braid $b_{\phi}  \in \SB_{2n}$ 
with $\Gamma(b_{\phi}) = \phi$. 
By abuse of notation, 
we regard $\phi$ as a representative of $\phi$. 
Notice that $\phi$ extends to a self-homeomorphism of 
$S^3 = B^+ \cup_S B^-$ 
preserving $\Bar{\mathcal{A}}^{d} $ and $ {}^{b}\mathcal{A}$ setwise. 
By Remark \ref{rem_extension} and Lemma~\ref{lem_trivial-tangle}, 
it follows that 
$$\Bar{\mathcal{A}}^{d} = \Bar{\mathcal{A}}^{d b_{\phi}^{-1}} \hspace{2mm}\mbox{and} \hspace{2mm}
{}^{b}{\mathcal{A}}={}^{b_{\phi} b}{\mathcal{A}}.$$
The second equality implies that 
$b_{\phi} \in \SW_{2n}({}^{b} \mathcal{A}) $.

Now, 
consider the image of $\Bar{\mathcal{A}}^{d}= \Bar{\mathcal{A}}^{d b_{\phi}^{-1}} $ under $\rho$. 
We have 
$${}^{d^{-1}} \mathcal{A} = \rho(\Bar{\mathcal{A}}^{d})= \rho(\Bar{\mathcal{A}}^{d b_{\phi}^{-1}})= 
 {}^{b_{\phi} d^{-1}} \mathcal{A} .$$
Hence we have 
$\mathcal{A} = {}^{d b_{\phi} d^{-1}} \mathcal{A}$, 
which implies that 
$ d b_{\phi} d^{-1} \in \SW_{2n}$. 
Thus $b_{\phi} \in d^{-1} ( \SW_{2n} ) d = \SW_{2n} ({}^{d^{-1}} \mathcal{A})$ 
by Lemma~\ref{lem_wicket-con}. 
Putting them together, we have 
$$b_{\phi} \in   \SW_ {2n}({}^{d^{-1}} \mathcal{A}) \cap \SW_{2n}({}^{b} \mathcal{A}) =  \SW_{2n}({}^{d^{-1}}\mathcal{A}, {}^{b}\mathcal{A}),$$ 
which says that 
$\phi \in \Gamma \bigl( \SW_{2n}({}^{d^{-1}}\mathcal{A}, {}^{b}\mathcal{A})\bigr)$. 
We are done.

The proof of 
$\mathcal{G}(\Bar{\mathcal{A}}^{d} \cup_S  {}^{b}\mathcal{A})  \supset 
\Gamma \bigl( \SW_{2n}({}^{d^{-1}}\mathcal{A}, {}^{b}\mathcal{A})\bigr)$ 
is similar. 
\end{proof}

Let $\Bar{\mathcal{A}}^{d} \cup_S  {}^{b}\mathcal{A}$ be an $n$-bridge decomposition of a link $L$ for some $b,d \in \SB_{2n}$. 
Then it is equivalent to the $n$-bridge decomposition 
$ \Bar{\mathcal{A}} \cup_S   {}^{d b}\mathcal{A}$. 
 By Lemma~\ref{lem_wicket-conj} and Theorem~\ref{thm_bridge_char}, 
one sees that their Goeritz groups $\mathcal{G}(\Bar{\mathcal{A}}^{d} \cup_S  {}^{b}\mathcal{A})$ and 
$\mathcal{G}( \Bar{\mathcal{A}} \cup_S   {}^{d b}\mathcal{A})$ are conjugate to each other in  $\MCG (\Sigma_{0, 2n})$.

\subsection{Relation to hyperelliptic Goeritz groups of Heegaard splittings}
\label{subsec:Relation to hyperelliptic Goeritz groups of Heegaard splittings}

Let $(M; \Sigma)= V^+ \cup_{\Sigma} V^-$ be a genus-$g$ Heegaard splitting 
with $g \geq 2$. 
Assume that there exists an involution 
$$\hat{\iota}: (M, V^+) \rightarrow (M, V^+) $$ 
such that 
$ \hat{\iota}|_{\Sigma}$ is a hyperelliptic involution on the Heegaard surface $\Sigma$. 
By definition 
$\hat{\iota}|_{V^+}$ and $\hat{\iota}|_{V^-}$ are 
hyperelliptic involutions of the handlebodies $V^+$ and $V^-$, respectively. 
Let $\SHomeo_+ (M, V^+)$ denote the centralizer in $\Homeo_+ (M, V^+)$ of $\hat{\iota}$. 
The {\it hyperelliptic Goeritz group} $\mathcal{HG}_{\hat{\iota}}(M; \Sigma) $ 
is then defined by 
$$\mathcal{HG}_{\hat{\iota}}(M; \Sigma)= \pi_0 (\SHomeo_+ (M, V^+)). $$
See Remark \ref{rem:notation for the hyperelliptic Goeritz groups} for this notation. 
Let $q: M \to M / \hat{\iota} = S^3$ be the projection. 
Set $B^\pm := q (V^\pm)$, $S := q (\Sigma)$, 
$L := q (\Fix (\hat{\iota})) $. 
We note that $(L; S) = (B^+, B^+ \cap L) \cup_S (B^-, B^- \cap L)$ 
is a $(g + 1)$-bridge decomposition of the link $L \subset S^3$.

The following theorem, which  implies Theorem~\ref{introthm:HG(M;V)/i}, 
is again a consequence of 
Theorem~\ref{thm:Birman-Hilden} and 
Lemma~\ref{lem:lifting a homeomorphism}  as in Theorem~\ref{thm:H(V)/i}. 

\begin{thm}
\label{thm:HG(M;V)/i}
The natural map 
$$ \mathcal{HG}_{\hat{\iota}}(M ;\Sigma) \to \mathcal{G}(L ; S) $$
is surjective and its kernel is $\langle [\hat \iota] \rangle$. 
Thus, we have 
$\mathcal{HG}_{\hat{\iota}}(M ; \Sigma) / \langle [\hat{\iota}] 
\rangle \cong \mathcal{G} (L; S)$. 
\end{thm}

We denote by  
$\mathrm{E}^\pm \Homeo_+ ( \Sigma )$ the subgroup of $\Homeo_+ (\Sigma)$ 
consisting of 
homeomorphisms of  $\Sigma$ that extend to those of $V^{\pm}$. 

The following theorem, which implies
Theorem~\ref{introthm:hyperelliptic Goeritz group as an intersection}, 
 corresponds to 
Theorem~\ref{thm:hyperelliptic handlebody group as an intersection} 
for hyperelliptic handlebody groups. 

\begin{thm}
\label{thm:hyperelliptic Goeritz group as an intersection} 
The map 
$$ \SHomeo_+ (M, V^+) \to 
\mathrm{E}^+ \Homeo_+ (\Sigma) \cap \mathrm{E}^- \Homeo_+ (\Sigma) \cap 
\SHomeo_+ (\Sigma) $$ 
that takes $\hat{f} \in \SHomeo_+ (M, V^+)$ to $\hat{f}|_{\Sigma}$ 
induces an isomorphism 
$$\mathcal{HG}_{\hat{\iota}} (M ; \Sigma) \xrightarrow{\cong} 
\mathcal{G} (M ; \Sigma) \cap \mathcal{H} ( \Sigma ).$$ 
\end{thm}
\begin{proof}
Recall that $\pi_0 (\SHomeo_+ (M, V^+)) 
= \mathcal{HG}_{\hat{\iota}} (M ; \Sigma)$.  
As in the proof of 
Theorem~\ref{thm:hyperelliptic handlebody group as an intersection}, 
we can think of 
$$ \pi_0 ( \mathrm{E}^+ \Homeo_+ (\Sigma) \cap \mathrm{E}^- \Homeo_+ (\Sigma) \cap 
\SHomeo_+ (\Sigma) ) $$ 
as the subgroup of 
$$ \pi_0 ( \mathrm{E}^+ \Homeo_+ (\Sigma) \cap \mathrm{E}^- \Homeo_+ (\Sigma) ) $$
consisting of the mapping classes that contain an element of 
$\SHomeo_+ ( \Sigma )$. 
Therefore, we can identify 
$$ \pi_0 ( \mathrm{E}^+ \Homeo_+ (\Sigma) \cap \mathrm{E}^- \Homeo_+ (\Sigma) \cap 
\SHomeo_+ (\Sigma) ) $$ 
with $\mathcal{G} (M ; \Sigma) \cap \mathcal{H} (\Sigma)$.

The surjectivity and injectivity of 
the map $\mathcal{HG}_{\hat{\iota}} (M ; \Sigma) \to 
\mathcal{G} (M ; \Sigma) \cap \mathcal{H} ( \Sigma )$
follow from Lemma~\ref{lem:Hirose-Kin 2017 Prop A.6}(1) and (2), respectively.  
\end{proof}

We can summarize the above discussion as follows. 
Let $L$ be a link in $S^3$ admitting an $n$-bridge decomposition 
$(L; S) = \Bar{\mathcal{A}} \cup_S {}^b \mathcal{A}$ for some $b \in \SW_{2n}$.  
Let  
$q: M_L \rightarrow S^3$ be the $2$-fold covering branched over $L$, and 
set $\Sigma := q^{-1} (S)$. 
The preimage of $q$  of the genus-$0$ Heegaard splitting $S^3= B^+ \cup_S B^-$  gives 
a genus-$(n - 1)$ Heegaard splitting 
$$(M_L; \Sigma )= q^{-1}(B^+) \cup_{\Sigma} q^{-1}(B^-).$$
We call  $(M_L; \Sigma)$ the {\it Heegaard splitting of} $M_L$ 
{\it associated with the bridge decomposition} $(L; S)$. 
Let 
$T: M_L \rightarrow M_L$ 
 be the non-trivial deck transformation of $q: M_L \rightarrow S^3$. 
We note that  
$T|_{\Sigma}: \Sigma \rightarrow \Sigma$
is a  hyperelliptic involution. 
Let $\mathcal{H}(\Sigma) $ be 
the hyperelliptic mapping class group associated with $T|_{\Sigma}$. 
By Theorems~\ref{thm_bridge_char}, 
\ref{thm:HG(M;V)/i} and \ref{thm:hyperelliptic Goeritz group as an intersection}, 
we have the following canonical identifications:
$$\mathcal{HG}_T(M_L; \Sigma) = 
 \mathcal{G}(M_L; \Sigma) \cap \mathcal{H}(\Sigma) = 
\Pi^{-1}(\Gamma(\SW_{2n} (\mathcal{A}, {}^b \mathcal{A}))) .$$
$$\mathcal{HG}_T(M_L; \Sigma) / \langle [T] \rangle 
 = \mathcal{G} (L; S) 
 = \Gamma ( \SW_{2n} (\mathcal{A}, {}^b \mathcal{A}) )  
= \SW_{2n} (\mathcal{A}, {}^b \mathcal{A}) / \langle \Delta^2 \rangle . $$

\begin{remark}
\label{rem:notation for the hyperelliptic Goeritz groups}
Let $(M; \Sigma)= V^+ \cup_{\Sigma} V^-$ be a genus-$g$ Heegaard splitting with $g \geq 2$. 
Assume that there exists an involution 
$$\hat{\iota}: (M, V^+) \rightarrow (M, V^+) $$ 
such that 
$ \hat{\iota}|_{\Sigma}$ is a hyperelliptic involution on the Heegaard surface $\Sigma$.  
Some readers might wonder why 
we write $\mathcal{HG}_{\hat{\iota}}(M; \Sigma) $ rather than $\mathcal{HG} (M; \Sigma) $, whereas 
a hyperelliptic mapping class group and a hyperelliptic handlebody group
are simply denoted by $\mathcal{H} (\Sigma_g)$ and $\mathcal{H} (V_g)$, respectively. 
The reason for the case of $\mathcal{H} (\Sigma_g)$ is that 
any two hyperelliptic involutions of a closed surface $\Sigma_g$ are conjugate in 
$\MCG  (\Sigma_g)$ $($see e.g. Farb-Margalit \cite[Proposition 7.15]{FarbMargalit12}$)$. 
Thus, any two hyperelliptic mapping class groups are conjugate. 
In particular, the structure of the group $\mathcal{H} (\Sigma_g)$ does not depend on the choice of 
a particular hyperelliptic involution of $\Sigma_g$. 
The same fact holds for hyperelliptic handlebody groups as well 
by Lemma $\ref{lem:Pantaolini-Piergallini}$. 
In the case of hyperelliptic Goeritz groups, however, the situation is more subtle. 
In fact, the conjugacy class of  
the above involution $\hat{\iota}$ in the Goeritz group 
does depend on the choice of $\hat{\iota}$ as we shall see now. 

Let $(H, S)$ be the $2$-bridge decomposition of the Hopf link $H \subset S^3$. 
Let $p, p'$ be points of $S \cap H$ that are contained in different components of $H$. 
Consider the two $4$-bridge decompositions $(H; S_{(p ,2)})$ and 
$ (H; (S_{(p ,1)})_{( p',1)})$ of $H$. 
Apparently, these bridge decompositions are not equivalent. 
Let $q: \mathbb{RP}^3 \rightarrow S^3$ be the $2$-fold covering branched over $H$. 
Let $(\mathbb{RP}^3; \Sigma_1)$ and  $(\mathbb{RP}^3; \Sigma_2)$ be the Heegaard splittings of $\mathbb{RP}^3$ 
associated with $(H; S_{(p ,2)})$ and $(H; (S_{( p ,1)})_{( p',1)})$, respectively. 
Let $T: \mathbb{RP}^3 \rightarrow \mathbb{RP}^3$ 
 be the non-trivial deck transformation of $q: \mathbb{RP}^3 \rightarrow S^3$. 
Then    
$T|_{\Sigma}: \Sigma \rightarrow \Sigma$ and 
$T|_{\Sigma'}: \Sigma' \rightarrow \Sigma'$  
are hyperelliptic involutions on $\Sigma$ and $\Sigma'$, respectively. 
Recall, by the way, that due to Bonahon-Otal \cite{BonahonOtal83}, $(\mathbb{RP}^3 ; \Sigma_1)$ and $(\mathbb{RP}^3 ; \Sigma_2)$
are equivalent. 
Therefore, for the unique genus-$3$ Heegaard splitting $(\mathbb{RP}^3, \Sigma)$ of $\mathbb{RP}^3$, 
we can consider the two involutions: one corresponds to $T$ for $(\mathbb{RP}^3, \Sigma)$ 
and the other corresponds to $T$ for $(\mathbb{RP}^3, \Sigma')$. 
We denote the former involution by $\hat{\iota}$ and the latter by $\hat{\iota}'$. 
Then $\hat{\iota}$ and $\hat{\iota}'$ can no longer be conjugate in 
the Goeritz group  $\mathcal{G}  (\mathbb{RP}^3; \Sigma)$, for 
$ (H; S_{(p,2)})$ and $ (H; (S_{(p,1)})_{(p',1)})$ are not equivalent. 
Therefore, we cannnot conclude that 
the hyperelliptic Goeritz groups
$\mathcal{HG}_{\hat{\iota}} (\mathbb{RP}^3; \Sigma)$ and 
$\mathcal{HG}_{\hat{\iota}'} (\mathbb{RP}^3; \Sigma)$ 
are conjugate. 
\end{remark}


Let $S$ be an $n$-bridge sphere of a link $L \subset S^3$. 
Take a point $p \in L \cap S$. 
Let 
$$(L; S_{(p,k)}) = (B^+_{(p,k)} \cap L) \cup_{S_{(p,k)}} (B^-_{(p,k)} \cap L) $$ 
be the bridge decomposition of $L$, 
where $S_{(p,k)}$ is the bridge sphere of $L$ obtained from the 
$k$-fold stabilization of $S$ near $p$. 
Consider the Heegaard splitting  
 $$q^{-1}(B^+_{(p,k)}) \cup_{q^{-1}(S_{(p,k)})}  q^{-1}(B^-_{(p,k)}) $$
of  $M_L$ associated with the bridge decomposition $(L; S_{(p,k)}) $. 
We note that the Heegaard surface $q^{-1}(S_{(p,k)})$ of $M_L$ 
is obtained from the Heegaard surface $q^{-1}(S)$ 
by $k$ stabilizations. 
To see this, take a disk $U$ embedded in 
the $3$-ball $B^+$ together with an arc $\gamma \subset \partial U$ 
(as in Section \ref{subsection_Bridge-decompositions}) 
so that $S_{(p,1)} = \partial (B^- \cup N(\gamma))$ is obtained from $S$ by a stabilization. 
Let $\widetilde{U} = q^{-1}(U)$, $\widetilde{\gamma}= q^{-1}(\gamma) $ be the preimages of $U$, $\gamma$, respectively. 
See Figure \ref{fig_stabilization}.
Then $\widetilde{\gamma}$ is a properly embedded arc in $  q^{-1} (B^+)$ 
which is parallel to $q^{-1}(S)$. 
Therefore, $\partial (q^{-1} (B^-) \cup N(\tilde{\gamma}))$ is obtained from $q^{-1}(S)$ 
by a stabilization. 
By iterating the same argument, we obtain the desired claim. 

\begin{center}
\begin{figure}[t]
\includegraphics[height=3cm]{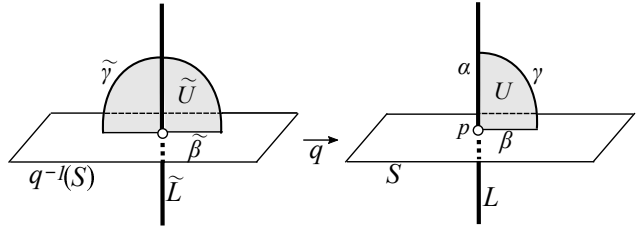}
\caption{$\widetilde{U} = q^{-1}(U)$ and $\widetilde{\gamma}= q^{-1}(\gamma) $.}
\label{fig_stabilization}
\end{figure}
\end{center}

\subsection{Examples}
\label{subsection_examples}
 
Below we provide several examples of the Goeritz groups of bridge decompositions 
that can be computed from the definitions. 
In Examples 
\ref{ex:Goeritz group of trivial knot} and 
\ref{ex:Goeritz group of Hopf knot} we describe the 
Goeritz groups of the $n$-bridge decomposition $(O; n)$ of the trivial knot $O$ 
and the $n$-bridge decomposition $(H; S_{p, n-2})$ of the Hopf link $H$ defined in 
Section \ref{subsection_Bridge-decompositions} in terms of wicket groups, 
which play a key role in Section \ref{section_application}. 
In Examples \ref{ex:Goeritz group of 2-bridge decomposition} and 
\ref{ex:Goeritz group of some 3-bridge decomposition}
we give explicit presentations of the Goeritz groups of 
2- and 3-bridge decompositions of all 2-bridge links. 
A sufficient condition for the Goeritz group to be {an infinite} group 
in terms of the distance will also be provided.

\begin{ex}
\label{ex:Goeritz group of n-component trivial link}
Consider  the $n$-bridge decomposition 
$(O_n ; n)= \Bar{\mathcal{A}} \cup_S \mathcal{A}$ of the $n$-component trivial link $O_n$, 
where $\Bar{\mathcal{A}}= \Bar{\mathcal{A}_n}$ and $\mathcal{A}= \mathcal{A}_n$. 
By Theorems~\ref{thm:Kin-Hirose Theorem 2.6} and \ref{thm_bridge_char} 
we have 
$$\mathcal{G}(O_n; n) \cong 
\Gamma \bigl( \SW_{2n}(\mathcal{A}, \mathcal{A})\bigr) = \Gamma(\SW_{2n}) 
= \MCG  (B^-, \mathcal{A}) .$$ 
In particular, the group $\mathcal{G}(O_n; n)$ is an infinite group except 
$\mathcal{G} (O_1; 1) \cong \Z / 2 \Z$. 
A finite generating set of $\mathcal{G} (O_n; n)$ is given by work \cite{Hilden75} of Hilden 
on $\MCG  (B^-, \mathcal{A})$. 
The asymptotic behavior of the minimal pseudo-Anosov dilatations in these groups was 
studied in \cite{HiroseKin17}. 
\end{ex}

\begin{ex}
\label{ex:Goeritz group of trivial knot}
Recall that the $n$-bridge decomposition 
$\Bar{\mathcal{A}} \cup_{S_{( p, n-1)}} \mathcal{B} $ 
defined in Section \ref{subsection_Bridge-decompositions} is 
the bridge decomposition $(O; n)$ of the trivial knot $O$. 
By Theorem~\ref{thm_bridge_char}, we have 
$\mathcal{G}(O; n)= \Gamma(\SW_{2n} (\mathcal{A}, \mathcal{B}))$. 
The $3$-sphere $S^3$ is the $2$-fold covering of $S^3$ branched over $O$. 
Let $(S^3; \Sigma)$ be the genus-$(n - 1)$ Heegaard splitting of $S^3$ 
associated with $(O; n)$. 
Then by Theorems \ref{thm:HG(M;V)/i} and 
\ref{thm:hyperelliptic Goeritz group as an intersection} 
we have 
$$\mathcal{HG}_T(S^3; \Sigma ) =  \Pi^{-1}(\Gamma(\SW_{2n} (\mathcal{A}, \mathcal{B}))).$$
\end{ex}

\begin{ex}
\label{ex:Goeritz group of Hopf knot}
Recall that the $n$-bridge decomposition 
$\Bar{\mathcal{A}} \cup_{S_{(p, n-2)}} \mathcal{C}$  
defined in Section \ref{subsection_Bridge-decompositions} 
is the bridge decomposition $(H; S_{(p, n-2 )})$ 
of the Hopf link $H$. 
By Theorem~\ref{thm_bridge_char}, we have 
$\mathcal{G}(H; S_{(p,  n - 2 )})= \Gamma(\SW_{2n}(\mathcal{A}, \mathcal{C}))$. 
The real projective space $\mathbb{RP}^3$ is 
the $2$-fold covering of $S^3$ branched over $H$.  
Let $(\mathbb{RP}^{3}; \Sigma )$ be the genus-$(n - 1)$ Heegaard splitting of 
$\mathbb{RP}^{3}$ associated with  
$(H; S_{(p, n - 2 )})$. 
Then by Theorems \ref{thm:HG(M;V)/i} and 
\ref{thm:hyperelliptic Goeritz group as an intersection} 
we have 
$$\mathcal{HG}_T(\mathbb{RP}^{3}; \Sigma) = 
\Pi^{-1}( \Gamma(\SW_{2n}(\mathcal{A}, \mathcal{C}))).$$

\end{ex}

\begin{ex}
\label{ex:Goeritz group of 2-bridge decomposition}
Let $L = S(p, r)$ be a $2$-bridge link (or the trivial knot) 
given by the Schubert normal form (Schubert \cite{Schubert56}, 
see also Hatcher-Thurston \cite{HatcherThurston85}{)}. 
By Otal \cite{Otal82}, there exists a unique 2-bridge decomposition 
$(L; S) = (B^+ \cap L) \cup_S (B^- \cap L)$ of $L$. 
If $(p, r ) = (0,1)$, then 
$L$ is the $2$-component trivial link $O_2$ and the Goeritz group 
$\mathcal{G} (L; S) = \mathcal{G} (O_2; 2)$ has 
already described in Example \ref{ex:Goeritz group of n-component trivial link}. 
Suppose that $(p, r ) \neq (0,1)$.  
 Since $(B^+ \cap L)$ (respectively, $(B^- \cap L)$) is a trivial $2$-tangle, 
there exists a unique essential separating disk $D^+$ 
(respectively, $D^-$) in $B^+ - (B^+ \cap L)$ 
(respectively, $B^- - (B^- \cap L)$). 
This implies that any element of $\mathcal{G} (L; S)$ preserves 
both $D^+$ and $D^-$.  
Since $(p,r) \neq (0,1)$, $S - (\partial D^+ \cup \partial D^-)$ consists only of disks, 
and each of them contains at most one point of $S \cap L$. 
Therefore, 
$\mathcal{G} (L; S)$ acts on the pair $(\partial D^+, \partial D^+ \cap \partial D^-)$ faithfully. 
This implies that $\mathcal{G} (L; S)$ is a subgroup of the dihedral group $D_k$, where 
$k = \# (\partial D^+ \cap \partial D^-) $. 
In fact, it is now easily checked that for any $(p,r) \neq (0,1)$, 
the Goeritz group  
$\mathcal{G} (L; S)$ is isomorphic to $ D_2 \cong \Z / 2 \Z \times \Z / 2 \Z$, 
where the generators are given in Figure \ref{fig_2-bridge}. 
In the figure, we think $L$ and $S$ as being embedded in $\R^3$, and 
the 3-ball bounded by $S$ is $B^-$. 
The element shown on the left-hand side in the figure is 
$\Gamma (\sigma_1^{-1} \sigma_3) $,  and 
the one shown on the right-hand side is $\Gamma (\Delta)$, 
where $\Gamma : \SB_4 \to \MCG (\Sigma_{0, 4}) $. 

\begin{center}
\begin{figure}[t]
\includegraphics[width=11cm]{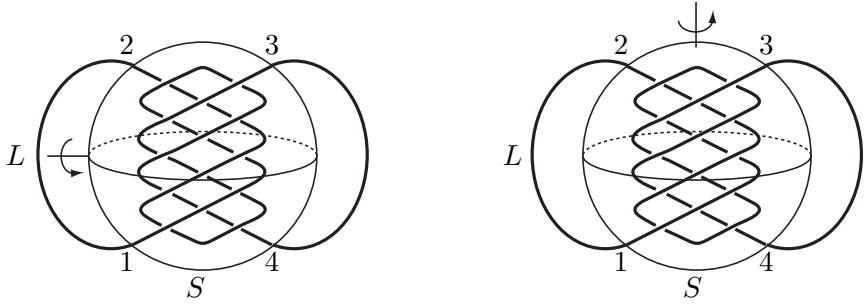}
\begin{picture}(400,0)(0,0)
\put(55,13){$1$}
\put(55,95){$2$}
\put(110,95){$3$}
\put(110,13){$4$}
\put(12,53){$L$}
\put(80,3){$S$}
\put(242,13){$1$}
\put(242,95){$2$}
\put(297,95){$3$}
\put(297,13){$4$}
\put(200,53){$L$}
\put(268,3){$S$}
\end{picture}
\caption{The two generators of $\mathcal{G} (L; S) \cong \Z / 2 \Z \times \Z / 2 \Z$ in the case of 
$L = S(5,2)$. 
They are given as half rotations along the illustrated axes.} 
\label{fig_2-bridge}
\end{figure}
\end{center}

\end{ex}

\begin{ex}
\label{ex:Goeritz group of some 3-bridge decomposition}
Let $L = S(p,  r)$ be again a $2$-bridge link (or the trivial knot) given by the Schubert normal form. 
Let $(L; S) = (B^+ \cap L) \cup_S (B^- \cap L)$ be 
a 3-bridge decomposition of $L$. 
Let $q: L( p , r ) \rightarrow S^3$ be 
the $2$-fold covering branched over $L$, 
where $L(p,r)$ is a lens space. 
Set $\Sigma := q^{-1} (S)$. 
Let $T$ be the non-trivial deck transformation of $q$. 
Then by Theorem~\ref{thm:HG(M;V)/i}, 
we have 
$\mathcal{G} (L; S) \cong \mathcal{HG}_{T} (L(p, r ); \Sigma ) / 
\langle [T] \rangle$.  
Since the genus of $\Sigma$ is two, 
it follows from Theorem~\ref{thm:hyperelliptic Goeritz group as an intersection} that 
the hyperelliptic Goeritz group 
$\mathcal{HG}_{T} (L(p,  r ); \Sigma)$ is canonically isomorphic to 
the Goeritz group $\mathcal{G} (L(p,  r ); \Sigma )$ itself, 
whose finite presentation is given 
by a series of work 
\cite{Akbas08, Cho08, Cho13, ChoKoda14, ChoKoda16, ChoKoda19}. 
Therefore, we can obtain a finite presentation of the Goeritz group 
$\mathcal{G} (L; S) \cong \mathcal{G} (L(p,  r ); \Sigma ) / \langle [T] \rangle$ 
for each $(p,r)$. 
\end{ex}

\begin{question}
Is the Goeritz group of the bridge decomposition of a link in $S^3$ always finitely generated? 
In particular, is $\mathcal{G} (O; n)$ finitely generated for any $n$?
\end{question}

\begin{ex}
\label{ex:distance and infinite order Goeritz groups}
Let $(L; S) = (B^+ , B^+ \cap L) \cup_S (B^+ , B^+ \cap L) $ be 
an $n$-bridge decomposition of a link $L \subset S^3$ 
with $n \geq 3$ and $d (L; S) \leq 1$.  
Then we can show that the Goeritz group $\mathcal{G}(L; S)$ is an infinite group as follows. 
Here we regard that $\mathcal{G}(L; S)$ is a subgroup of 
$\MCG  (S, S \cap L)$ consisting of elements that extend to 
both $\MCG  (B^+ , B^+ \cap L)$ and  $\MCG  (B^- , B^- \cap L)$. 
For convenience, we will not distinguish 
curves and homeomorphisms from their isotopy
classes. 
Similar arguments for Heegaard splittings can be found in 
Johnson-Rubinstein 
\cite[Corollary 6.2]{JohnsonRubinstein13} and 
Namazi \cite[Proposition 1]{Namazi07}. 

Suppose first that $d (L; S) = 0$. 
Then there exists an essential simple closed curve 
$\alpha \in \mathcal{D}^+ \cap \mathcal{D}^-$, that is, 
$\alpha$ is a simple closed curve on 
$S_L = \Cl (S - N (S \cap L; S))$ 
bounding a disk $D^+ \subset B^+ - (B^+ \cap L) $ and a disk 
$D^- \subset B^- - ( B^- \cap L) $.  
Therefore, the Dehn twist $\tau_\alpha$ about $\alpha$ is an element of 
$\mathcal{G} (L; S)$. 
Indeed, $\tau_\alpha$ extends to an element of $\mathcal{G} (L; S)$ as a 
rotation along the sphere $D^+ \cup D^-$. 
Since $\alpha$ is essential, 
the order of $\tau_\alpha$ in $\MCG  (S, S \cap L)$ is infinite. 
Thus $\mathcal{G}(L; S)$ is an infinite group. 

\begin{center}
\begin{figure}[t]
\includegraphics[width=8cm]{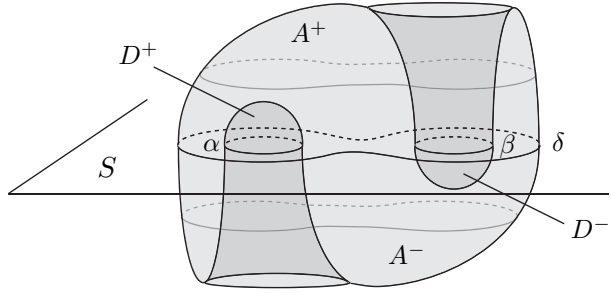}
\begin{picture}(400,0)(0,0)
\put(100,55){$S$}
\put(279,32){\begin{small}$D^-$\end{small}}
\put(108,98){\begin{small}$D^+$\end{small}}
\put(210,22){\begin{small}$A^-$\end{small}}
\put(173,106){\begin{small}$A^+$\end{small}}
\put(140,65){\begin{small}$\alpha$\end{small}}
\put(252,64){\begin{small}$\beta$\end{small}}
\put(272,65){\begin{small}$\delta$\end{small}}
\end{picture}
\caption{The 2-sphere $X = D^+ \cup A^- \cup A^+ \cup D^-$.} 
\label{fig_sphere_X}
\end{figure}
\end{center}

Suppose that  $d (L; S) = 1$. 
Then there exist disjoint essential simple closed curves 
$\alpha \in \mathcal{D}^+$ and $\beta \in \mathcal{D}^-$. 
(Note that here we use the assumption $n \geq 3$. 
Indeed, in the case of $n = 2$ the definition of the curve graph 
$\mathcal{C} ( S_L  ) $ is different from the usual case.) 
The simple closed curve $\alpha$ bounds a disk $D^+ \subset B^+ - (B^+ \cap L) $,  
and $\beta$ bounds a disk $D^- \subset B^- - (B^- \cap L) $. 
Take a simple arc $\gamma \subset S_L $ connecting $\alpha$ 
and $\beta$. 
Let $\delta$ be the component of the boundary of 
$N (\alpha \cup \gamma \cup \beta; S)$ that is not isotopic to 
$\alpha$ or $\beta$. 
Then $\alpha$ and $\delta$ cobounds an annulus $A^- \subset B^- - (B^- \cap L)$,  
while $\beta$ and $\delta$ cobounds an annulus $A^+$ in $D^+ \subset B^+ - (B^+ \cap L) $. 
In this way, we obtain a 2-sphere $X := D^+ \cup A^- \cup A^+ \cup D^-$ in 
$S^3 - L $ with $X \cap S = \alpha \cup \beta \cup \delta$, see Figure \ref{fig_sphere_X}.  
Consider the map $\phi := \tau_{\alpha} \circ \tau_{\delta}^{-1} \circ \tau_{\beta} : (S, S \cap L) \to (S, S \cap L)$. 
By the above construction, $\phi$ extends to a homeomorphism of 
$B^+$ as the composition of 
the twist about $D^+$ and the twist about $A^+$, 
while $\phi$ extends to a homeomorphism of 
$B^-$ as the composition of 
the twist about $D^-$ and the twist about $A^-$. 
Thus, $\phi$ extends to an element $\hat{\phi}$ of 
$\mathcal{G} (L; S)$. 
Since $\alpha$, $\beta$, $\delta$ are pairwise 
disjoint, pairwise non-parallel, essential simple closed curves on 
$S_L $, the order of $\phi$ in $\MCG  (S, S \cap L)$ is infinite. 
Therefore, $\mathcal{G}(L; S)$ is an infinite group. 
\end{ex}

\section{The Goeritz groups of high distance bridge decompositions}
\label{section_distance}

As we have seen in Example 
\ref{ex:distance and infinite order Goeritz groups}, 
the Goeritz group of a bridge decomposition $(L; S)$ is 
an infinite group if the distance of $(L; S)$ is at most one. 
In contrast to Example 
\ref{ex:distance and infinite order Goeritz groups}, 
we are going to show that the Goeritz group of $(L; S)$ 
is a finite group if the distance of 
$(L; S)$ is sufficiently large. 
The aim of this section is to prove 
Theorem~\ref{introthm:finite Goeritz group}, which is restated below. 

\begin{thm}
\label{thm:finite Goeritz group} 
There exists a uniform constant $N$ such that 
if the distance of an $n$-bridge decomposition 
$(L; S)$ of a link $L$ in $S^3$ with $n \geq 3$ is greater than $N$, 
then the Goeritz group $\mathcal{G}(L;S)$ is a finite group.   
\end{thm}

We note that an analogous fact was proved for Heegaard splittings 
by Namazi \cite{Namazi07}, that is, in that paper he showed that if the distance of a Heegaard splitting 
is sufficiently large, its Goeritz group is a finite group. 
If the distance of the Heegaard splitting associated with a bridge decomposition 
of high distance is also high, then Namazi's result together with 
Theorems \ref{thm:HG(M;V)/i} and \ref{thm:hyperelliptic Goeritz group as an intersection} 
immediately implies Theorem \ref{thm:finite Goeritz group}.  
However, we do not know at present whether 
there exists a lower bound of the distance of the associated Heegaard 
splitting in terms of that of a bridge decomposition.  
Instead, we will give a more direct proof, which shares the same spirit as 
Namazi's proof.  
We also note that Ohshika-Sakuma \cite{OhshikaSakuma16} studied 
another kind of groups related to Heegaard splittings and bridge decompositions. 
The main result of \cite{Namazi07} and the above theorem are comparable 
with Theorem 2 of \cite{OhshikaSakuma16}.

\begin{lem}
\label{lem:reducible case}
If $\mathcal{G}(L;S)$ contains a non-periodic reducible element, 
then the distance $d(L;S)$ is at most $2K+4$. 
Here $K$ is the uniform constant in Theorem~$\ref{thm:convexity}$. 
\end{lem}

\begin{proof}
Assume that $\mathcal{G}(L;S)$ contains non-periodic reducible element 
$\phi$. 
Let $\gamma$ be a curve in the canonical reducing system for $\phi$. 
\begin{claim}
\label{claim:reducingcurve}
Let $\alpha$ be an element of $\mathcal{C}^{0}(S_{L})$, 
where we recall that $S_L = \Cl (S - N (S \cap L) ; S)$. 
If $k>0$ is sufficiently large, then  
the distance between $\gamma$ and any geodesic segment that connects $\alpha$ and $\phi^{k}(\alpha)$ 
is at most $2$. 
\end{claim}

\begin{proof}[Proof of Claim $\ref{claim:reducingcurve}$]
Let $\{\gamma_{1},\ldots,\gamma_{s}\}$ be the canonical reducing system for $\phi$. 
Note that $d_{\mathcal{C}(\Sigma)}(\gamma,\gamma_{i}) \le 1$ for $1 \le i \le s$. 
By Namazi \cite{Namazi07}, 
there exists an essential subsurface $Y$ of $S_{L}$, 
where $Y$ is a pseudo-Anosov component of $\phi$ or an annular neighborhood of some $\gamma_{i}$, 
such that 
$d_{\mathcal{C}(Y)}(\pi_{Y}(\alpha),\pi_{Y}(\phi^{k}(\alpha))) \rightarrow \infty$ as $k \rightarrow \infty$. 
Let $c$ be a geodesic segment connecting $\alpha$ and $\phi^{k}(\alpha)$.  
By Theorem~\ref{thm:geodesic image},
if every vertex of $c$ intersects $Y$,  
then we have $\mathrm{diam}_{\mathcal{C}(Y)}( \pi_Y(c) ) \le C$.  
Here $C>0$ is the constant in Theorem~\ref{thm:geodesic image}.
Since $d_{\mathcal{C}(Y)}(\pi_{Y}(\alpha),\pi_{Y}(\phi^{k}(\alpha))) \rightarrow \infty$ as $k \rightarrow \infty$,  
there exists a vertex of $c$ that does not intersect $Y$ for a sufficiently large $k$. 
Thus the distance between $\partial Y$ and $c$ is at most $1$ in the curve 
graph $\mathcal{C}( S_L )$. 
Therefore the distance between $\gamma$ and $c$ is at most $2$.
\end{proof}

Let $\alpha$ be an arbitrary element of $\mathcal{D}^{+}$. 
Let $k>0$ be a sufficiently large integer. 
Let $c$ be a geodesic segment that connects $\alpha$ and $\phi^{k}(\alpha)$. 
By Theorem~\ref{thm:convexity}, 
$c$ lies within the $K$-neighborhood of $\mathcal{D}^{+}$. 
Combining this fact and Claim \ref{claim:reducingcurve}, 
we conclude that the distance between $\gamma$ and $\mathcal{D}^{+}$ is at most $K+2$. 
 
Since the same argument can be applied to $\mathcal{D}^{-}$,  
the distance between $\gamma$ and $\mathcal{D}^{-}$ is at most $K+2$.  
Thus we have  
$d(L;S) \le d_{\mathcal{C}(S_{L})}(\mathcal{D}^{+},\gamma)+d_{\mathcal{C}(S_{L})}(\gamma, \mathcal{D}^{-}) \le 2K+4$. 
\end{proof}

\begin{lem}
\label{lem:pA case}
If $\mathcal{G}(L;S)$ contains a pseudo-Anosov element, 
then the distance $d(L;S)$ is at most $2K+2\delta$. 
Here $K$ and $\delta$ are the uniform constants in 
Theorem~$\ref{thm:convexity}$ and Theorem~$\ref{thm:hyperbolicity}$, respectively. 
\end{lem}

\begin{proof}
Let $\alpha$ (respectively, $\beta$) be an arbitrary element of 
$\mathcal{D}^{+}$ (respectively, $\mathcal{D}^{-}$). 
Let $c_{k}$ be a geodesic segment connecting $\alpha$ and $\phi^{k}(\alpha)$ for each $k>0$. 
Let $d_{k}$ be a geodesic segment connecting $\beta$ and $\phi^{k}(\beta)$ for each $k>0$. 
By Theorem~\ref{thm:hyperbolicity}, 
the distance between $c_{k}$ and $d_{k}$ is at most $2\delta$ 
when $k$ is sufficiently large. 
By Theorem~\ref{thm:convexity}, 
$c_{k}$ lies within the $K$-neighborhood of $\mathcal{D}^{+}$. 
Similarly, 
$d_{k}$ lies within the $K$-neighborhood of $\mathcal{D}^{-}$. 
Thus we have 
$$d(L;S) \le 
d_{\mathcal{C}(S_{L})}(\mathcal{D}^{+},c_{k})+d_{\mathcal{C}(S_{L})}(c_{k},d_{k})+d_{\mathcal{C}(S_{L})}(d_{k}, \mathcal{D}^{-}) 
\le 2K+2\delta, $$ 
which is our assertion. 
\end{proof}

\begin{lem}
\label{lem:torsion}
Any torsion subgroup of $\MCG  (\Sigma_{0, 2n})$ is finite. 
\end{lem}

\begin{proof}
By Serre \cite{Serre61}, $\mathrm{Mod}(\Sigma_{n-1})$ contains a torsion-free subgroup $G$ of finite index.  
Set $G^{\prime}:=\mathcal{H}(\Sigma_{n-1}) \cap G$. 
Then we have 
$$[\mathcal{H}(\Sigma_{n-1}):G^{\prime}]
=[\mathcal{H}(\Sigma_{n-1}) \cdot G:G]
\le [\mathrm{Mod}(\Sigma_{n-1}):G]
< \infty.$$ 
Thus, the index of $G_{0}:=\Pi(G^{\prime})$ in $\MCG  (\Sigma_{0, 2n})$ is finite, 
where we recall that $\Pi : \mathcal{H} (\Sigma_{n-1}) \to \MCG  (\Sigma_{0, 2n})$ 
is the natural map. 
Since $G'$ is torsion free and $\ker \Pi \cong \Z / 2 \Z$, $G_0$ is also torsion free. 

Suppose that $F$ is a torsion subgroup of $\MCG  (\Sigma_{0, 2n})$. 
Since $G_0$ is torsion free, $F \cap G_{0}$ is the trivial group. 
Since $[\MCG  (\Sigma_{0, 2n}) : G_0 ]$ is finite, 
we conclude that $F$ is a finite group. 
\end{proof}

\begin{proof}[Proof of Theorem~$\ref{thm:finite Goeritz group}$]
Set $N:= \max \{ 2K + 4,  2K+2\delta \} $, where we recall this is a uniform constant. 
Let $(L;S)$ be an $n$-bridge decomposition of a link in $S^3$ with $n \geq 3$. 
Suppose that $d(L;S) > N$. 
By Lemmas~\ref{lem:reducible case} and \ref{lem:pA case},
$\mathcal{G}(L;S)$ contains neither a reducible element nor a pseudo-Anosov element. 
Thus $\mathcal{G}(L;S)$ is a torsion subgroup of $\MCG  (\Sigma_{0, 2n})$. 
By Lemma~\ref{lem:torsion}, 
$\mathcal{G}(L;S)$ is a finite group. 
\end{proof}
 
As we have explained in Sections \ref{subsec:Curve graphs} and 
\ref{subsec:The distance of bridge decompositions}, 
the constant $\delta$ can be chosen to be at most $102$, and 
$K$ can be chosen to be at most $1796$. 
Therefore, the above proof shows that 
the constant $3796$ is enough for the constant $N$ in Theorem~$\ref{thm:finite Goeritz group}$.

\section{Pseudo-Anosov elements in the Goeritz groups of stabilized bridge decompositions}
\label{section_stabilized}

It follows immediately from 
Example $\ref{ex:distance and infinite order Goeritz groups}$ and 
Lemma $\ref{lem: distance stabilized}$ that 
the Goeritz group  of a stabilized bridge decomposition of a link in $S^3$ is an infinite group except the case of 
the $2$-bridge decomposition $(O; 2)$ of the trivial knot $O$. 
For each of those bridge decompositions, we can find an infinite order element of the Goeritz group 
looking at a local part of the decomposition as follows. 
Let $(L; S)$ be an $n$-bridge decomposition of $L$ with $n \geq 2$. 
Let $p$ be a point in $S \cap L$. 
Without loss of generality, we may assume that the point $p$ is labeled by $2n$. 
Consider the $(n+1)$-bridge decomposition 
$$(L; S_{(p, 1)}) = (B^+_{(p, 1)} \cap L) \cup_{S_{(p,1)}} (B^-_{(p,1)} \cap L).$$  
Set $S':= S_{(p,1)}$. 
Recall that the triples $(S^3, S, L)$ and $(S^3, S', L)$ are identical 
except within a small $3$-ball $B$ near the point $p$ shown in Figure $\ref{fig_stabilized_goeritz}$$(1)$. 
Set $\alpha := \partial B \cap S'$. 
Since $\alpha$ bounds a disk $D^+ \subset B^+$ $($$D^- \subset B^-$, respectively$)$ 
with $\# (D^+ \cap L) = 1$ $($$\# (D^- \cap L) = 1$, respectively$)$, 
the Dehn twist $\tau_\alpha : (S', S' \cap L) \to (S' , S' \cap L)$ 
extends to an element of $\MCG ( B^+_{(p, 1)} , B^+_{(p, 1)} \cap L)$ 
$($$\MCG ( B^+_{(p, 1)} , B^+_{(p, 1)} \cap L)$, respectively$)$. 
Therefore, $\tau_\alpha \in \MCG (\Sigma_{0, 2n+2})$ defines an element $\hat{\tau}_{\alpha}$ of 
$\mathcal{G}(L;  S'  )$, whose order is clearly infinite. 
Note that $\tau_\alpha = \Gamma (b)$, where $b$ is  the element of $\SW_{2n+2}$ shown in 
Figure $\ref{fig_stabilized_goeritz}$$(2)$.
\begin{center}
\begin{figure}[t]
\includegraphics[width=12.5cm]{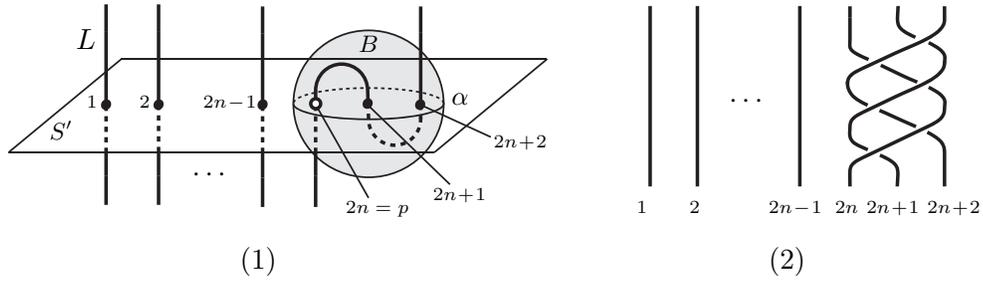}
\begin{picture}(400,0)(0,0)
\put(28,83){$L$}
\put(32,61){\begin{tiny}$1$\end{tiny}}
\put(52,61){\begin{tiny}$2$\end{tiny}}
\put(76,61){\begin{tiny}$2n \! - \! 1$\end{tiny}}
\put(130,21){\begin{tiny}$2n=p$\end{tiny}}
\put(163,26){\begin{tiny}$2n \! + \! 1$\end{tiny}}
\put(186,46){\begin{tiny}$2n \! + \! 2$\end{tiny}}
\put(170,62){\begin{small}$\alpha$\end{small}}
\put(72,33){$\cdots$}
\put(18,49){\begin{footnotesize}$S'$\end{footnotesize}}
\put(135,82){\begin{footnotesize}$B$\end{footnotesize}}
\put(275,61){$\cdots$}
\put(240,21){\begin{tiny}$1$\end{tiny}}
\put(260,21){\begin{tiny}$2$\end{tiny}}
\put(290,21){\begin{tiny}$2n \! - \! 1$\end{tiny}}
\put(315,21){\begin{tiny}$2n$\end{tiny}}
\put(327,21){\begin{tiny}$2n \! + \! 1$\end{tiny}}
\put(350,21){\begin{tiny}$2n \! + \! 2$\end{tiny}}
\put(90,0){$(1)$}
\put(290,0){$(2)$}
\end{picture}
\caption{(1) {A stabilized $(n+1)$-bridge decomposition of a link $L$ and the} $3$-ball $B${.} 
(2) An element $b \in \SW_{2n+2}$ (which is a ``full twist" with 3 strands).}
\label{fig_stabilized_goeritz}
\end{figure}
\end{center}
 
The infinite-order elements of the Goeritz group of a stabilized bridge decomposition 
we have given so far are all reducible: each of them is 
an extension of either a single Dehn twist (Figure \ref{fig_stabilized_goeritz}) 
or the composition of the 
Dehn twists about three disjoint simple closed curves 
(Example $\ref{ex:distance and infinite order Goeritz groups}$) 
in the bridge sphere. 
In this section, we  discuss 
pseudo-Anosov elements in that Goeritz group. 
In fact, we prove Theorem~\ref{introthm_pA elements in a stabilized bridge decomposition}, 
which is restated below.

\begin{thm}
\label{thm_pA elements in a stabilized bridge decomposition}
Let $(L; S)$ be an $n$-bridge decomposition of a link $L$ in $S^3$ 
with $n \geq 2$. 
Let $p$ be an arbitrary point in $S \cap L$. 
If $(L; S) = (O_2; 2)$, then the Goeritz group $\mathcal{G} (L; S_{(p,1)})$ 
is an infinite group consisting only of reducible elements. 
Otherwise, the Goeritz group $\mathcal{G} (L; S_{(p,1)})$ 
contains a pseudo-Anosov element. 
\end{thm}
 
There are two ingredients for the construction of 
a pseudo-Anosov element in the above theorem. 
One is a slight modification of the element given 
in the first paragraph of this section, which corresponds to 
a Dehn twist about a simple closed curve in $S_L$. 
The other is a construction of pseudo-Anosov elements by Penner \cite{Penner88}. 

In the following arguments, we always assume that 
curves under consideration in a (marked) surface 
are properly embedded, 
and their intersection is transverse and minimal up to isotopy. 
We will not distinguish 
curves, surfaces and homeomorphisms from their isotopy
classes in this section. 

Let 
$$(L; S) = (B^+ \cap L) \cup_{S} (B^- \cap L)$$  
be an $n$-bridge decomposition of a link $L$ in $S^3$ with 
$n \geq 2$. 
Let $p$ be an arbitrary point in $S \cap L$. 
Then there exists a unique component $T^+$ ($T^-$, respectively) of 
$B^+ \cap L$ ($B^- \cap L$, respectively) one of whose 
endpoints is $p$. 

\begin{center}
\begin{figure}[t]
\includegraphics[width=8cm]{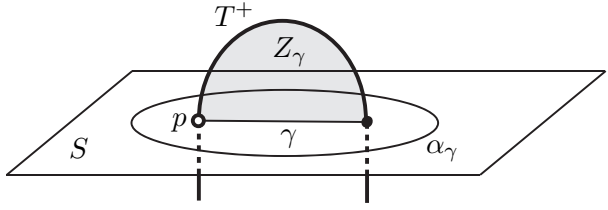}
\begin{picture}(400,0)(0,0)
\put(145,80){$T^+$}
\put(167,68){$Z_\gamma$}
\put(170,36){$\gamma$}
\put(129,42){$p$}
\put(90,30){$S$}
\put(225,32){$\alpha_\gamma$}
\end{picture}
\caption{A reference arc $\gamma$ for $T^+$.}
\label{fig_reference_arc}
\end{figure}
\end{center}

A simple arc $\gamma$ in the marked sphere $(S, S \cap L)$  
with $\partial \gamma = \partial T^+$ is called 
a {\it reference arc} for $T^+$ if 
there exists a disk $Z_\gamma$ embedded in $B^+$ 
such that $Z_\gamma \cap L = \partial Z_\gamma \cap L = T^+$ and $\partial Z_\gamma - T^+ = \gamma$. 
A simple closed curve $\alpha \in \mathcal{D}^+$ is said to be {\it associated with} $p$ 
if there exists a reference arc $\gamma$ for $T^+$ with 
$\alpha = \partial N (\gamma ; S)$. 
In this case,  we write $\alpha = \alpha_\gamma$. 
See Figure \ref{fig_reference_arc}. 
We note that $Z_\gamma$ and $\alpha_\gamma$ as above are uniquely determined for each $\gamma$. 
We denote by $\mathcal{D}_p^+$ the subset of $\mathcal{D}^+$ consisting of 
simple closed curves associated with $p$. 
The subset $\mathcal{D}_p^- \subset \mathcal{D}^-$ is defined exactly in the same way as above 
(using ``$-$" instead of ``$+$"). 
Two simple closed curves $\alpha , \beta \in \mathcal{C} (S_L)$ are said to {\it fill} the surface $S_L$ if 
the union $\alpha \cup \beta$ cuts $S_L$ into 
open disks and half-open annuli. 


\begin{lem}
\label{lem_filling pair}
If $(L; S) \neq (O_2; 2)$, then there exist simple closed curves 
$\alpha^{+} \in \mathcal{D}_p^+$ and 
$\alpha^{-} \in \mathcal{D}_p^-$ 
that fill $S_L$. 
\end{lem}
\begin{proof}
Suppose first that $n=2$, that is, $(L; S)$ is a 2-bridge decomposition. 
In this case, we have $\mathcal{D}_p^+ = \mathcal{D}^+$ ($\mathcal{D}_p^- = \mathcal{D}^-$, respectively), 
and it consists of only one 
simple closed curve $\alpha^{+} $ 
($\alpha^{-}$, respectively) 
(cf. Example \ref{ex:Goeritz group of 2-bridge decomposition}). 
If $(L; S) \neq (O_2; 2)$, then using its Schubert normal form 
it is easily seen that $\alpha^{+}$ and 
$\alpha^{-}$ fill $S_L$. 
(In the case $(L; S) = (O_2; 2)$, we have 
$\alpha^{+} = \alpha^{-}$ 
and they do not fill $S_L$.) 

In the following we suppose that $n \geq 3$. 
Choose arbitrary simple closed curves 
$\alpha^{+} \in \mathcal{D}_p^+$ and $\beta \in \mathcal{D}_p^-$. 
By \cite[Proposition 1.3]{HiroseKin17} there exists a pseudo-Anosov element $\phi$ in 
$\MCG(B^+, B^+ \cap L)$. 
By replacing $\phi$ with some positive power, if necessary, we can assume that 
$\phi$ is the identity on $\partial B^+ \cap L$. 
It follows from Masur-Minsky \cite[Proposition 4.6]{MasurMinsky99} that 
$$ \lim_{k \to \infty} d_{\mathcal{C} (S_L)} 
(\alpha^{+} , \phi^k ( \alpha^{+} ))  =\infty, $$ 
which in particular implies that 
there exists $k_0 \in \N$ with $d_{\mathcal{C} (S_L)} 
(\alpha^{+} , \phi^{k_0} ( \alpha^{+} )) \geq 5$. 
Since we have assumed that $\phi$ fixes $\partial B^+ \cap L$, 
the image
$\phi^{k_0} ( \alpha^{+} )$ remains to be contained in 
$\mathcal{D}_p^+$. 
Now by applying triangle inequalities we have $d_{\mathcal{C} (S_L)} 
(\alpha^{+} , \alpha^{-} ) \geq 3$ or 
$d_{\mathcal{C} (S_L)} (\phi^{k_0} ( \alpha^{+} ) , 
\alpha^{-}) \geq 3$, which implies the assertion. 
\end{proof}

\begin{center}
\begin{figure}[t]
\includegraphics[width=13cm]{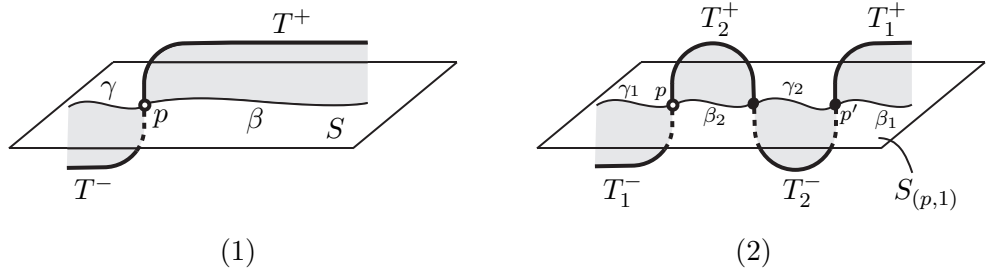}
\begin{picture}(400,0)(0,0)
\put(100,86){$T^+$}
\put(25,23){$T^-$}
\put(55,52){$p$}
\put(120,46){$S$}
\put(90,50){$\beta$}
\put(35,63){$\gamma$}
\put(262,88){$T_2^+$}
\put(225,23){$T_1^-$}
\put(325,88){$T_1^+$}
\put(293,23){$T_2^-$}
\put(242,62){\begin{scriptsize} $p$ \end{scriptsize}}
\put(312,53){\begin{scriptsize} $p'$ \end{scriptsize}}
\put(228,63){\begin{scriptsize} $\gamma_1$ \end{scriptsize}}
\put(325,51){\begin{scriptsize} $\beta_1$ \end{scriptsize}}
\put(290,64){\begin{scriptsize} $\gamma_2$ \end{scriptsize}}
\put(260,52){\begin{scriptsize} $\beta_2$ \end{scriptsize}}
\put(335,23){$S_{(p,1)}$}
\put(80,0){(1)}
\put(275,0){(2)}
\end{picture}
\caption{(1) The bridge decomposition $(L; S)$ around $p$. 
(2) The bridge decomposition $(L; S_{(p,1)})$ around $p$. }
\label{fig_stabilized_pA}
\end{figure}
\end{center}

\begin{center}
\begin{figure}[t]
\includegraphics[width=10cm]{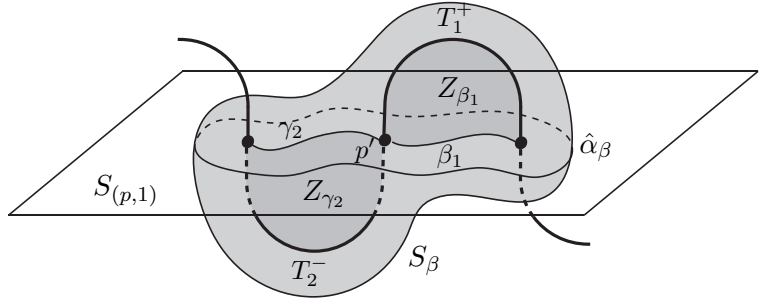}
\begin{picture}(400,0)(0,0)
\put(70,52){$S_{(p,1)}$}
\put(169,66){\begin{small}$p'$\end{small}}
\put(200,116){\begin{small}$T_1^+$\end{small}}
\put(145,20){\begin{small}$T_2^-$\end{small}}
\put(200,64){\begin{small}$\beta_1$\end{small}}
\put(140,75){\begin{small}$\gamma_2$\end{small}}
\put(200,89){$Z_{\beta_1}$}
\put(149,50){$Z_{\gamma_2}$}
\put(254,68){$\hat{\alpha}_\beta$}
\put(189,25){$S_\beta$}
\end{picture}
\caption{The sphere $S_\beta$. }
\label{fig_stabilized_pA2}
\end{figure}
\end{center}

\begin{proof}[Proof of Theorem $\ref{thm_pA elements in a stabilized bridge decomposition}$]
By a 1-fold stabilization of $(L; S)$ at $p$, we obtain a bridge decomposition 
$$(L; S_{(p, 1)}) = (B_{(p, 1)}^+ \cap L) \cup_{S_{(p, 1)}} (B_{(p, 1)}^- \cap L) . $$  

Suppose first that $(L; S) \neq (O_2; 2)$. 
By Lemma \ref{lem_filling pair} 
there exist reference arcs $\beta$ and $\gamma$ for $T^+$ and $T^-$, respectively, such that 
simple closed curves 
$\alpha_\beta \in \mathcal{D}_p^+$ and $\alpha_\gamma \in 
\mathcal{D}_p^-$ 
fill $S_L$. 
Let $T_1^+$ and $T_2^+$ ($T_1^-$ and $T_2^-$, repsectively) be the components of 
$B_{(p, 1)}^+ \cap L$ ($B_{(p, 1)}^{-} \cap L$, respectively), 
$p'$ the point of $S_{(p, 1)} \cap L$ as shown in Figure \ref{fig_stabilized_pA}. 
Since the component $T_1^+$ ($T_1^-$, respectively) of $B_{(p, 1)}^+ \cap L$ 
($B_{(p, 1)}^- \cap L$, respectively) naturally corresponds to 
$T^+$ ($T^-$, respectively), 
the reference arc $\beta$ ($\gamma$, respectively)  
defines in a canonical way 
a reference arc $\beta_1$ ($\gamma_1$, respectively) for $T_1^+$ ($T_1^-$, respectively). 
Here we note that under the convention in Section \ref{subsection_Bridge-decompositions} 
(Figure \ref{fig_sphere-stab}), 
$\gamma_1$ is nothing but $\gamma$ thought of as being embedded in $S_{(p, 1)}$. 
Choose a reference arc $\beta_2$ ($\gamma_2$, respectively) for $T_2^+$ ($T_2^-$, respectively) 
 so that $\beta_2 \cap \gamma_1 = \{ p \}$ ($\gamma_2 \cap \beta_1 = \{ p' \}$, respectively). 
 Consider the 2-spheres $S_{\beta} := \partial N (Z_{\beta_1} \cup Z_{\gamma_2})$ and 
 $S_{\gamma} := \partial N (Z_{\gamma_1} \cup Z_{\beta_2})$, see Figure \ref{fig_stabilized_pA2}. 
Since $\alpha_\beta$ and $\alpha_\gamma$ fill the surface $S_L$, 
the simple closed curves 
$\widehat{\alpha}_\beta := S_{(p,1)} \cap S_{\beta}$ and 
$\widehat{\alpha}_{\gamma} := S_{(p,1)} \cap S_{\gamma}$ fill $(S_{(p,1)})_L$. 
Set $D_{\beta}^+ := S_{\beta} \cap B_{(p, 1)}^+$, $D_{\beta}^- := S_{\beta} \cap B_{(p, 1)}^-$, 
$D_{\gamma}^+ := S_{\gamma} \cap B_{(p, 1)}^+$ and $D_{\gamma}^- := S_{\gamma} \cap B_{(p, 1)}^-$. 
Then each of the disks $D_{\beta}^+$, $D_{\beta}^-$, $D_{\gamma}^+$ and $D_{\gamma}^-$ 
intersects $L$ once and transversely.  
This implies that each of the Dehn twists 
$\tau_{\widehat{\alpha}_\beta}$ and 
$\tau_{\widehat{\alpha}_\gamma}$ 
extends to an element of $\mathcal{G}(L; S_{(p,1)})$. 
Now, it follows from Penner \cite{Penner88} that the composition 
$\tau_{\widehat{\alpha}_\beta} \circ 
\tau_{\widehat{\alpha}_\gamma}^{-1}$ gives rise to a pseudo-Anosov element 
of $\mathcal{G} (L; S)$.   

Finally, suppose that $(L; S) = (O_2; 2)$. 
Let $q: S^2 \times S^1 \rightarrow S^3$ be 
the $2$-fold covering branched over $L$. 
Set $\Sigma := q^{-1} (S_{(p,1)})$, which is the unique genus-$2$ Heegaard splitting of $S^2 \times S^1$. 
Let $T$ be the non-trivial deck transformation of $q$. 
Then as we have seen in Example 
\ref{ex:Goeritz group of some 3-bridge decomposition}, 
the group 
$\mathcal{G} (L; S)$ is isomorphic to $\mathcal{G} (S^2 \times S^1 ; \Sigma ) / 
\langle [T] \rangle$.  
The conclusion now follows from Cho-Koda \cite{ChoKoda14}, which shows 
that the Goeritz group 
$\mathcal{G} (S^2 \times S^1 ; \Sigma )$ is an infinite group 
consisting only of reducible elements. 
\end{proof}

\section{Asymptotic behavior of minimal pseudo-Anosov entropies}
\label{section_application}

In this section, we prove Theorems~\ref{introthm_asymptitic behavior for the trivial knot}
and \ref{introthm_asymptotic behavior for the Hopf link}. 

We say that $\beta \in \SB_n$ is {\it pseudo-Anosov} 
if $ \Gamma(\beta) \in \MCG (\Sigma_{0,n})$ is  a pseudo-Anosov mapping class. 
In this case, 
the {\it dilatation} $\lambda(\beta)$ 
(respectively, {\it entropy} $\log \lambda(\beta)$) of $\beta$ is 
defined by the dilatation (respectively, entropy) of  $\Gamma(\beta)$.

Collapsing $\partial D$ to a point, 
we obtain an injective homomorphism 
$$\mathfrak{c}: \MCG (D_n) \rightarrow \MCG (\Sigma_{0, n+1}).$$
We say that $b \in B_n$ is {\it pseudo-Anosov} if 
$ \mathfrak{c}(\Gamma_D(b)) $ is a pseudo-Anosov mapping class. 
Then the {\it dilatation} $\lambda(b)$ 
(respectively, {\it entropy} $\log \lambda(b)$) of $b$ is 
defined by the dilatation (respectively, entropy) of $ \mathfrak{c}(\Gamma_D(b)) $. 
Let 
$$s: B_n \rightarrow \SB_n$$ 
be the surjective homomorphism 
which sends a braid $b \in B_n$ to the spherical braid in $\SB_n$ represented by the same word of letters $\sigma_j^{\pm 1}$'s as $b$. 
Let 
$$s_+: B_n \rightarrow \SB_{n+1}$$ 
be the homomorphism 
which sends a braid $b \in B_n$ to the spherical braid obtained from $s(b)$ with $n$ strands 
adding the $(n+1)$th straight strand. 
Hence $s_+(b)$ is also represented by the same word of letters $\sigma_j^{\pm 1}$'s as $b$. 

\begin{remark}
\label{rem_s+}
By the definition of pseudo-Anosov braids in $B_n$, we see that 
$b \in B_n$ is pseudo-Anosov if and only if 
$s_+(b) \in \SB_{n+1}$ is pseudo-Anosov. 
In this case, $\lambda(b)= \lambda(s_+(b))$ holds. 
\end{remark}

For the proofs of Theorems~\ref{introthm_asymptitic behavior for the trivial knot}
and \ref{introthm_asymptotic behavior for the Hopf link}, 
we use a result in \cite{HiroseKin18}, which we recall now. 
Let $z_n $ be a pseudo-Anosov braid on the plane with $d_n$ strands. 
We say that a sequence $\{z_n\}$ {\it has a small normalized entropy} 
if   $d_n \asymp n$ and 
there is a constant $P>0$ which does not depend on $n$ such that 
\begin{equation}
\label{equation_upper-bound}
d_n \cdot  \log \lambda(z_n) \le P . 
\end{equation}
It is known by Penner \cite{Penner91} that 
if $\phi \in \MCG (\Sigma_{0,n})$ is pseudo-Anosov for $n \ge 4$, then 
$\log \lambda(\phi) \ge \frac{\log 2}{4n-12}$. 
This implies that 
 if $\{z_n\}$ satisfies (\ref{equation_upper-bound}), then 
  we have 
 $$\log \lambda(z_n) \asymp \frac{1}{n}.$$

For the definition of {\it $i$-increasing planar braids}, 
 see \cite[Section 3]{HiroseKin18}. 
 ($i$ stands for the indices of strands.)
 If a pseudo-Anosov braid $b$ is $i$-increasing, then 
one obtains a pseudo-Anosov braid  $(b \Delta^{2n})_1$ with more strands than $b$ for each $n \ge 1$ 
which is well-defined up to conjugate \cite[Section 4.1]{HiroseKin18}. 
The number of strands of $(b \Delta^{2n})_1$ can be computed from $b$. 
The braid $(b \Delta^{2n})_1$ enjoys the property such that 
the mapping torus of $(b \Delta^{2n})_1$ is homeomorphic to the mapping torus of the original  braid $b$ \cite[Example 4.4(3)]{HiroseKin18}. 
Furthermore, the sequence of pseudo-Anosov braids $\{(b \Delta^{2n})_1\}$ varying $n$ has a small normalized entropy 
\cite[Theorem 5.2(3)]{HiroseKin18}.

\begin{center}
\begin{figure}[t]
\includegraphics[height=3cm]{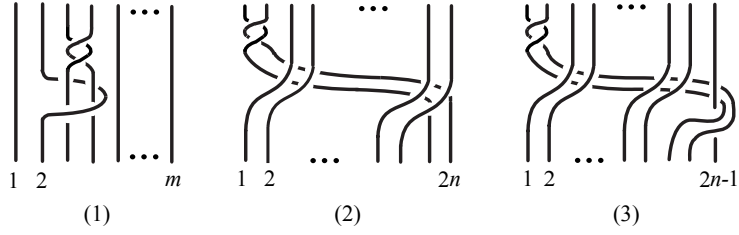}
\caption{
(1) $X= X_m \in B_{m}$. 
(2) $Y =Y_{2n} \in B_{2n}$. 
(3) $Z= Z_{2n-1} \in B_{2n-1}$.} 
\label{fig_geneB}
\end{figure}
\end{center}

For each $m \ge 5$ and $n \ge 3$, 
we define the braids $X,Y$ and $Z$ as follows. 
\begin{eqnarray*}
X= X_m &=&  (\sigma_2 \sigma_{3})^3 = \sigma_{3}^2 \sigma_2 \sigma_{3}^2 \sigma_2 \in B_m, 
\\
Y= Y_{2n} &=&  \sigma_1^2 \sigma_2 \sigma_3 \cdots \sigma_{2n-1} \sigma_1 \sigma_2 \cdots \sigma_{2n-2} \in B_{2n}, 
\\
Z= Z_{2n-1} &=& \sigma_1^2 \sigma_2 \sigma_3 \cdots \sigma_{2n-2} \sigma_1 \sigma_2 \cdots \sigma_{2n-3} \sigma_{2n-3} \sigma_{2n-2} \in B_{2n-1}, 
\end{eqnarray*}
see Figure \ref{fig_geneB}. 
The spherical $6$-braids 
$s_+(X_5)$, $s(Y_6)$ and $s_+(Z_5)$ are equal to  $x$, $y$, and $z$, respectively,  in Example \ref{ex_SWA}. 
For each $n \ge 3$, 
we have 
$s(Y), s_+(Z) \in \SB_{2n}$. 
For each $m= 2n-1$ with $n \ge 3$, we have 
$s_+(X) \in \SB_{2n}$. 
We write 
$$x = s_+(X), \ y = s(Y),\ z= s_+(Z) \in \SB_{2n}.$$ 
It is easy to see the following lemma. (Cf. Lemma \ref{lem_SWBC}.)

\begin{lem} 
\label{lem_xyz}
We have $x,y \in  \SW_{2n}(\mathcal{A}, \mathcal{B})$ 
and $x,z \in  \SW_{2n}(\mathcal{A}, \mathcal{C})$. 
\end{lem}

For a subgroup $G \subset \SB_n$ containing a pseudo-Anosov element, 
we write $\ell(G) = \ell(\Gamma(G))$. 
Example \ref{ex:Goeritz group of trivial knot} tells us that 
$$\ell(\mathcal{G}(O; n))= \ell(\SW_{2n} (\mathcal{A}, \mathcal{B})).$$
We can then restate Theorem~\ref{introthm_asymptitic behavior for the trivial knot}
as follows. 
\begin{thm}
\label{thm_ABwicket}
We have 
$$\ell (\mathcal{G}(O; n))  =\ell(\SW_{2n}({\mathcal{A}}, {\mathcal{B}})) \asymp \dfrac{1}{n}.$$
\end{thm}

Theorem~\ref{thm_ABwicket} implies Corollary \ref{introcor_sphere-entropy}. 
The reason is that 
if $\phi \in \MCG (\Sigma_{0, 2g+2 })$ is pseudo-Anosov, then 
each element of $\Pi^{-1}(\phi)$ is pseudo-Anosov with the same entropy as $\phi$. 
Theorem~\ref{thm_ABwicket} together with 
Example \ref{ex:Goeritz group of trivial knot} says that 
$$\ell(\mathcal{HG}_T(S^3; q^{-1}(S_{(p, g )})))  \asymp \frac{1}{  g }.$$ 
Since 
$$\mathcal{HG}_T(S^3; q^{-1}(S_{(p, g )})) \subset \mathcal{G}(S^3; g) \subset \MCG (\Sigma_g)$$ 
and 
$\ell(\MCG (\Sigma_g)) \asymp \frac{1}{g}$ by Penner \cite{Penner91}, 
we conclude that 
$$\ell(\mathcal{G}(S^3; g)) \asymp \frac{1}{g}.$$

\begin{proof}[Proof of Theorem~$\ref{thm_ABwicket}$]
In the proof, we regard $D_n$  as the disk with $n$ punctures. 
For braids $X, Z \in B_5$ as above, we consider the product 
$$\alpha:= X Z=
(\sigma_3^2 \sigma_2 \sigma_3^2 \sigma_2) (\sigma_1^2 \sigma_2 \sigma_3 \sigma_4 \sigma_1 \sigma_2 \sigma_3^2 \sigma_4) \in B_5,$$
see Figure \ref{fig_gseedtwist}(1). 
We now claim that  $\alpha $ is pseudo-Anosov. 
To see this, we first observe that 
$$\alpha \Delta^{-2}= \sigma_2 \sigma_3 \sigma_3 \sigma_2 
\sigma_4^{-1} \sigma_3^{-1} \sigma_3^{-1} \sigma_4^{-1} \sigma_2^{-1} \sigma_1^{-1} \sigma_3^{-1} \sigma_2^{-1}. $$
Let $\gamma$ denote the following $5$-braid 
$$\gamma= \sigma_1 \sigma_2 \sigma_4^{-1} \sigma_3^{-1} \sigma_3^{-1} \sigma_4^{-1} \sigma_2^{-1} \sigma_3^{-1}.$$
Then one can check that 
$\eta^{-1} (\alpha \Delta^{-2}) \eta \gamma^{-1}$ equals the identity element in $B_5$, 
where 
$$\eta=  \sigma_1 \sigma_2 \sigma_1 \sigma_3 \sigma_2 \sigma_1 \sigma_4 \cdot \sigma_1 \sigma_2 \sigma_1 \sigma_4 \sigma_3 \sigma_2 \sigma_1 \cdot 
\sigma_2 \sigma_3 \sigma_2 \sigma_1 \sigma_4 \sigma_3 \cdot \sigma_3 \sigma_4 \sigma_3 \sigma_2 \sigma_1.$$
%
%
%
This implies that $\Gamma_D(\alpha) = \Gamma_D(\alpha \Delta^{-2})$ is conjugate to $\Gamma_D(\gamma)$ in $\MCG (D_5)$.  
It is enough to show that $\gamma $ is pseudo-Anosov, 
for the Nielsen-Thurston types are preserved under the conjugation.

Remove the third and fourth strands from $\gamma$, 
we obtain  $\sigma_1 \sigma_2^{-2} \in B_3$. 
It is easy to see that $\sigma_1 \sigma_2^{-2} $ is  pseudo-Anosov. 
For instance, see \cite{Handel97}. 
Let $\Phi: D_3 \rightarrow D_3$ be a pseudo-Anosov homeomorphism 
which represents $\Gamma_D(\sigma_1 \sigma_2^{-2}) \in \MCG (D_3)$. 
Let $\mathcal{O}$ be a periodic orbit with period $k$ of $\Phi$. 
 Blow up each  periodic point in $\mathcal{O}$. 
Then we still have a pseudo-Anosov homeomorphism 
$\Phi^{\circ}:  D_3 \setminus \mathcal{O} \rightarrow D_3 \setminus \mathcal{O}$ 
defined on the $(3+k)$-punctured disk with the same entropy as $\Phi$. 
By using train track maps for pseudo-Anosov $3$-braids (see \cite{Handel97}), 
it is not hard to show that 
there is a periodic orbit $\mathcal{O}$ with period $2$ 
such that 
%
the pseudo-Anosov homeomorphism $\Phi^{\circ}:  D_3 \setminus \mathcal{O} \rightarrow D_3 \setminus \mathcal{O}$ 
 represents  $\Gamma_D(\gamma)  \in \MCG (D_5) $. 
Thus $\gamma$ is pseudo-Anosov.

Next, we  prove that 
$\ell(\SW_{4n+2}({\mathcal{A}}, {\mathcal{B}})) \asymp \frac{1}{n}$. 
We note that  the above $\alpha \in B_5$ is a $5$-increasing braid. 
One sees that  $(\alpha \Delta^{2n})_1$ is written by 
$$(\alpha \Delta^{2n})_1 =   X Z^{2n+1}\in B_ {4n+7}.$$
(See Figure \ref{fig_gseedtwist}(2) in the case $n=1$.)  
Let  $(\alpha \Delta^{2n})_1^{\bullet} $ be the braid with $(4n+6)$ strands obtained from $(\alpha \Delta^{2n})_1 $  
by removing the last strand. 
Then we have 
$$(\alpha \Delta^{2n})_1^{\bullet} = X Y^{2n+1} \in B_{4n+6}.$$
(See Figure \ref{fig_gseedtwist}(3) in the case $n=1$.)  
By Lemma~\ref{lem_xyz}, we have 
$$s((\alpha \Delta^{2n})_1^{\bullet}) = s( X Y^{2n+1})= xy^{2n+1} \in \SW_{4n+6}({\mathcal{A}}, {\mathcal{B}}).$$
Then \cite[Lemma 6.3]{HiroseKin18} tells us that for $n$ large, 
$s((\alpha \Delta^{2n})_1^{\bullet}) $ is a  pseudo-Anosov braid with the same entropy as 
$(\alpha \Delta^{2n})_1 $. 
Therefore we have 
$$\log \lambda (s((\alpha \Delta^{2n})_1^{\bullet}))= \log \lambda(x y^{2n+1}) \asymp \frac{1}{n}, $$ 
since 
the sequence $\{(\alpha \Delta^{2n})_1\}$ varying $n$ has a small normalized entropy \cite[Theorem 5.2(3)]{HiroseKin18}, 
and we are done.

Finally, we  prove $\ell(\SW_{4n}({\mathcal{A}}, {\mathcal{B}})) \asymp \frac{1}{n}$.  
We consider $\alpha^2 = X Z X Z \in B_5$ which is pseudo-Anosov, since so is $\alpha$. 
This is a  $5$-increasing braid as well. 
We consider the sequence of pseudo-Anosov braids 
$(\alpha^2 \Delta^{2n})_1$ varying $n$. 
The braid  $(\alpha^2 \Delta^{2n})_1$ can be written by 
$(\alpha^2 \Delta^{2n})_1 = X Z X Z^{2n+1} \in B_ {4n+9}$. 
Then  $(\alpha^2 \Delta^{2n})_1^{\bullet} \in B_ {4n+8}$ obtained from $(\alpha^2 \Delta^{2n})_1$ 
by removing the last strand 
is of the form 
$(\alpha^2 \Delta^{2n})_1^{\bullet} = X Y X Y^{2n+1} \in B_{4n+8}$. 
Hence its spherical braid satisfies the following property. 
$$s((\alpha^2 \Delta^{2n})_1^{\bullet}) = s(X Y X Y^{2n+1} ) = x y x y^{2n+1} \in \SW_{4n+8}({\mathcal{A}}, {\mathcal{B}}).$$
By  \cite[Lemma 6.3]{HiroseKin18} again, it follows that 
for $n$ large, 
$s((\alpha^2 \Delta^{2n})_1^{\bullet})$ is a pseudo-Anosov braid with the same entropy as 
$(\alpha^2 \Delta^{2n})_1$.  
The sequence of pseudo-Anosov braids 
$(\alpha^2 \Delta^{2n})_1$ has a small normalized entropy, and hence 
this property also holds for $s((\alpha^2 \Delta^{2n})_1^{\bullet})= x y  x y^{2n+1} \in \SW_{4n+8}  ({\mathcal{A}}, {\mathcal{B}})$. 
This completes the proof. 
\end{proof}

\begin{center}
\begin{figure}[t]
\includegraphics[height=8cm]{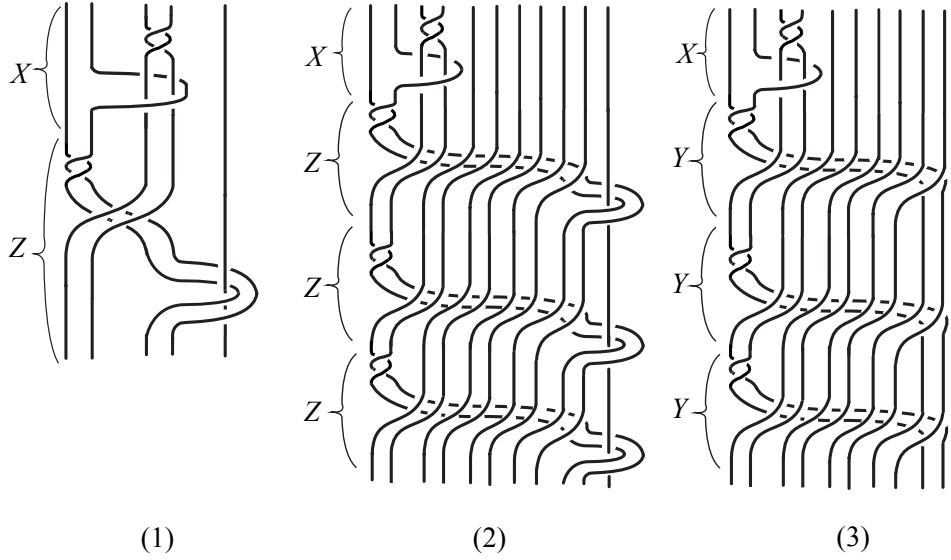} 
\caption{
(1) $\alpha= X_5Z_5 \in B_5$. 
(2) $(\alpha \Delta^2)_1 = X_{11} Z_{11}^3 \in B_{11}$. 
(3) $(\alpha \Delta^2)_1^{\bullet} = X_{10} Y_{10}^3 \in B_{10}$.}
\label{fig_gseedtwist}
\end{figure}
\end{center}

Finally we prove Theorem~\ref{introthm_asymptotic behavior for the Hopf link}. 
By {Example \ref{ex:Goeritz group of Hopf knot}}, 
we have 
$$\ell(\mathcal{G}(H; S_{(p, n - 2)}))= \ell(\SW_{2n}(\mathcal{A}, \mathcal{C})).$$
We restate Theorem~\ref{introthm_asymptotic behavior for the Hopf link}
as follows. 
\begin{thm}
\label{thm_ACwicket}
We have 
$$\ell(\mathcal{G}(H; S_{(p, n-2 )}))= \ell(\SW_{2n}({\mathcal{A}}, {\mathcal{C}})) \asymp \dfrac{1}{n}.$$
\end{thm}
As in the proof of Corollary \ref{introcor_sphere-entropy}, 
Corollary \ref{introcor_rp3-entropy} then follows from Theorem \ref{thm_ACwicket} and Example 
\ref{ex:Goeritz group of Hopf knot}.

\begin{proof}[Proof of Theorem~$\ref{thm_ACwicket}$]
We first prove that 
$\ell(\SW_{4n}({\mathcal{A}}, {\mathcal{C}})) \asymp \frac{1}{n}$.  
In the proof of Theorem~\ref{thm_ABwicket}, 
we obtain a sequence of pseudo-Anosov braids 
$(\alpha \Delta^{2n})_1 =   X Z^{2n+1}\in B_ {4n+7}$ having a small normalized entropy.  
We have  
$$s_+((\alpha \Delta^{2n})_1)=s_+( X Z^{2n+1}) = xz^{2n+1} \in \SB_{4n+8},$$ 
and it is  an element of $ \SW_{4n+8}({\mathcal{A}}, {\mathcal{C}})$ by Lemma~\ref{lem_xyz}. 
Since $s_+((\alpha \Delta^{2n})_1)$ is a pseudo-Anosov braid with the same entropy as $(\alpha \Delta^{2n})_1$ (see Remark \ref{rem_s+}), 
we are done. 

Next, we prove $\ell(\SW_{4n+2}({\mathcal{A}}, {\mathcal{C}})) \asymp \frac{1}{n}$. 
In the proof of Theorem~\ref{thm_ABwicket}, 
we obtain a sequence of pseudo-Anosov braids 
$(\alpha^2 \Delta^{2n})_1 = X Z X Z^{2n+1} \in B_ {4n+9}$ 
having a small normalized entropy. 
By Remark \ref{rem_s+}, 
the spherical braid $s_+( X Z X Z^{2n+1} ) \in \SB_{4n+10}$ 
has the same entropy as $X Z X Z^{2n+1}  \in B_{4n+9}$. 
We have  
$$s_+( X Z X Z^{2n+1} ) = xzxz^{2n+1}$$ 
which is an element of $\SW_{4n+10}({\mathcal{A}}, {\mathcal{C}})$ 
by Lemma~\ref{lem_xyz}. 
This completes the proof. 
\end{proof}
 
Theorems \ref{introthm_asymptitic behavior for the trivial knot} and 
\ref{introthm_asymptotic behavior for the Hopf link} 
motivate us to pose the following question. 

\begin{ques}
For any bridge decomposition $(L; S)$ of a link $L \subset S^3$  and any point $p \in L \cap S$  
do we have $\ell(\mathcal{G}(L; S_{(p,k)})) \asymp \frac{1}{k}$, 
where $S_{(p,k)}$ is the bridge sphere of $L$ obtained from a $k$-fold stabilization of $S$ at $p$?
\end{ques}
 
\section*{Acknowledgments} 
The authors wish to express their gratitude to Makoto Ozawa, Masatoshi Sato,  
Robert Tang and Kate Vokes for many helpful comments.

%
%

\medskip
\bibliographystyle{alpha} 
\bibliography{bridge}

\end{document}